 \documentclass[a4, 12pt, reqno]{amsart}

\usepackage{amssymb}
\usepackage{amstext}
\usepackage{amsmath}
\usepackage{amscd}
\usepackage{latexsym}
\usepackage{amsfonts}
\usepackage{color}
\usepackage{textcomp}
\usepackage{graphicx}
\usepackage[all]{xy}
\usepackage[latin1]{inputenc}
\usepackage{enumitem}
\theoremstyle{plain}
\newtheorem*{thm*}{Theorem}
\newtheorem*{cor*}{Corollary}
\newtheorem*{defn*}{Definition}
\newtheorem*{claim*}{Claim}

\newtheorem{theorem}{Theorem}[section]
\newtheorem{corollary}[theorem]{Corollary}
\newtheorem{lemma}[theorem]{Lemma}

\newtheorem{example}[theorem]{Example}

\newtheorem{definition}[theorem]{Definition}
\newtheorem{remark}[theorem]{Remark}

\newcommand*{\defeq}{\mathrel{\vcenter{\baselineskip0.5ex \lineskiplimit0pt
                     \hbox{\scriptsize.}\hbox{\scriptsize.}}}%
                      =}
\theoremstyle{definition}

\theoremstyle{remark}

\def\max{\mathrm{max}}

\def\N{\mathbb N}

\def\Z{\mathbb Z}

\def\card{\mathrm{Card}}

\def\In{\mathrm{in}}

\def\Ap{\mathrm{Ap}}

\def\fz{\mathbb{Z}}
\def\fn{\mathbb{N}}

\def\bn{{\mathbb N}}

\def\min#1{{\rm min}\,\{#1\}}

\def\card{{\rm card}}

\def\gbb{Gr\"obner basis\ }
\def\gb{Gr\"obner\ }

\begin{document}

\setlength{\baselineskip}{14.5pt}
\title[Symmetric and Almost Symmetric   semigroups]{Frobenius number of Almost Symmetric  numerical generalized almost arithmetic semigroups}

\author{ Marcel Morales}
\address{Institut Fourier, UMR 5582, Laboratoire de Mathématiques,
Université Grenoble Alpes, CS 40700, 38058 Grenoble cedex 9, France}
\email{ marcel.morales@univ-grenoble-alpes.fr}
 
\author{Nguyen Thi Dung}
\address{University of Information and Communication Technology, Thai Nguyen, Vietnam}
\email{ntdung.cntt@ictu.edu.vn}

\thanks{  The second author  was supported partially by the Institut Fourier, UMR 5582, Laboratoire de Mathématiques,
Université Grenoble Alpes, CS 40700, 38058 Grenoble cedex 9, France}

\begin{abstract} Let  $a,  k,h, c$ be positive integers and $d$ a non zero integer. Recall that a {\it numerical generalized almost arithmetic semigroup }  $S$ is a semigroup minimally generated by relatively prime  positive integers $a, ha+d, ha+2d, \ldots, ha+kd, c, $ that is its embedding dimension is $k+2.$ In a previous work, the authors  described the Apéry set and  a \gbb of the ideal defining $S$ under one technical assumption, the complete version will be published in a forthcoming paper. In this paper we continue with this assumption and  we describe the Pseudo Frobenius set. As a consequence we give a complete description of  $S$ when  it is symmetric or almost symmetric as well as  generalize and extend the previous results of  Ignacio García-Marco, J. L. Ram\'irez Alfons\'in and O. J. R\o dseth; we also find  a quadratic formula for its Frobenius number that  generalizes some results of J.C. Rosales, and P.A. García-Sánchez. Moreover, for given numbers $a, d, k,h, c,$ a simple algorithm allows us to determine if $S$ is almost symmetric or not and furthermore to find its type and  Frobenius number.
\end{abstract}
\maketitle
\tableofcontents
 
 \section{Introduction} 

Let  $a_0,\ldots, a_n$ be natural numbers  and $S\defeq\{k_0 a_0+\ldots + k_n a_n \mid  k_i\in \bn\}$ the  
semigroup  generated by $\{a_0,\ldots,a_n\}.$  Recall that if $a_0,\ldots, a_n$ are  relatively prime numbers then the {\it Frobenius number} of $S$, denoted  by $F(S),$ is the biggest integer that does not belong to $S$.
 Let
 $K\lbrack S\rbrack\defeq K[t^k \mid  k\in S]=K[t^{a_0}, \ldots , t^{a_n}]\subset K[t]$ be the {\it  semigroup ring} of $S$ and $R=K[x_{0},\ldots,x_{n}]$   the polynomial ring  in $n+1$ variables over $K$  graded by the weights  $\deg x_i=a_i,$ for all $i=0, \ldots, n.$ The {\it defining ideal} $I$ of $K[S]$ is defined to be  the kernel of the $K$-algebra homomorphism $\Psi \colon  R  \rightarrow K[S]$  given by $\Psi (x_i)=t^{a_i} $ for all  $i=0,\ldots, n$. We will use often the fact that $I$ is a prime ideal generated by binomials and does not contains monomials. We use the  {\it   weighted degree reverse  lexicographical order} $\prec _{w}$  on the set of monomials $[[R]]$ of the ring $R$  with  $x_0\prec \ldots\prec  x_n,$ and {\it the map} $\varphi\colon  [[R]]\rightarrow \bn$  defined by  $\varphi(M)=k_0a_0+ \ldots + k_n a_n$ for every monomial $M=x_0^{k_0}\ldots x_n^{k_n} \in [[R]]$. Let us recall that the {\it pseudo-Frobenius set} $PF(S)$ is the set of all  integer numbers  $a$ which satisfies $ a\notin S$ and $a+s\in S,$ for all $0\ne s\in S$, the number of elements of $PF(S)$ is called the {\it type} of $S$ and is  denoted by $t(S).$  Finally, the {\it Ap\'ery set}, which is defined by   ${\rm Ap}(S,a_0)\defeq\{ s\in S \mid s-a_0 \notin S\}$ plays an important role in our paper. By defining in \cite{M-D}, a monomial Ap\'ery set $\widetilde {\Ap(S,a_0)}$ with respect to $a_0$  as an algebraic analogous to the   Ap\'ery set $\Ap(S,a_0)$ and using the  order $\prec _{w}$  as well as  the map $\varphi$, we can change from studying the Ap\'ery set to studying the set of monomials in $[[R']]$   which are not in $\In(I)$, where  $R'=K[x_1,\ldots,x_n]$ and $\In(I)$ is the monomial ideal of the leading terms of $I$ with respect to  $\prec _{w}$.
 
 Let  $a,  k,h, c$ be positive integers and $d\in \fz^*$. Recall that a {\it numerical generalized almost arithmetic semigroup} (AAG-semigroup for short) is a semigroup minimally generated by relatively prime positive integers $a, ha+d, ha+2d, \ldots, ha+kd, c, $ that is its embedding dimension is $k+2.$ Our goal is to describe all properties of an AAG-semigroup in terms of a continuous fraction, as an extension of our previous works in \cite{Mo1}, \cite{Mo2}, \cite{M-D}. It  was given in \cite{M-D} a \gbb of the ideal defining $S$ under one assumption and the complete case will be published in a forthcoming paper. In this paper we continue the work of \cite{M-D}, with the same  assumption  we can describe the Pseudo Frobenius set (see Theorem \ref{tPSF}). As a consequence we  give a complete description of AAG-semigroups that are symmetric or almost symmetric (see Theorems \ref{tGor}, \ref{almostsym0}, \ref{almostsym1}, \ref{almostsym2}) which  improves and extends  the results of \cite{GRR}. In particular we prove that if $S$ is almost symmetric then its  type is at most the embedding dimension minus 1. Another interesting point is that if $S$ is almost symmetric then the Frobenius number is given by a linear or quadratic formula in terms of $a, d, k,h, c$  and $t(S)$. Moreover, a simple algorithm using the solutions of some quadratic equations allows us to decide if an AAG-semigroup is almost symmetric  or not and furthermore we can find its type and  Frobenius number. This result  extends and generalizes all the results of \cite{R-G}.

The algorithms  presented here are the extensions of the previous works  by the author in  \cite{Mo1}, \cite{Mo2} and can be downloaded in  http: //www-fourier.univ-grenoble-alpes.fr/~morales/.

\section{Preliminaries}
 

Denote by $\Z$  and $\N$ the set of integers and nonnegative integers respectively.
Given $n\geqslant 1$
and $a_0,\ldots,a_n\in \N$ such that $\gcd(a_0,\ldots,a_n)=1$, let $S$ be the semigroup in $\N$
$$S\defeq\{k_0 a_0+\ldots k_n a_n \mid k_i\in \bn\}.$$
The set $\N\setminus S$ is finite and $S$ is called  {\it numerical semigroup}. If $S$  is minimally generated by  $\{a_0, \ldots,a_n\}$ then   $n+1$ is called the {\it embedding dimension} of $S.$

\begin{definition} \label {d2} {\rm Let $S$ be a numerical semigroup generated by $a_0,\ldots ,a_n$.

{(i)} The number $F(S)=\max \{a\in \Z \mid a\notin S\}$ is called  the {\it Frobenius number} of $S$.

 {(ii)} We also define
$\displaystyle PF(S)\defeq\{a\in \Z\setminus S\mid  a+s \in S \text {\ if \ } s \in S \text {\ and\ } s\ne 0\}$ the pseudo Frobenius set.
An element of $PF(S)$ is called a {\it pseudo-Frobenius number} of $S$. Obviously, the Frobenius number  
is a pseudo-Frobenius number and the number of elements of $PF(S)$ is called the {\it type} of $S$, denoted by $t(S).$
 
{(iii)} The {\it Ap\'ery set } with respect to $a_0$ in  $S$  is the set
${\rm Ap}(S,a_0)=\{ s\in S \mid  s-a_0 \notin S\}.$
}
\end{definition}

 A numerical  semigroup is { symmetric} if $t(S)=1$. We will use the following property to characterize almost symmetric numerical semigroups.
 
 \begin{lemma}\label {Nari}\cite{N} Suppose $t(S)>1$. Let $PF(S)=\{f_1,\dots, f_{t(S)-1}, F(S)\}$ with $f_1<\dots<f_{t(S)-1}$. Then $S$ is almost symmetric if and only if
 $f_i+ f_{t(S)-i}= F(S)$ for $i=1,\dots,t(S)-1$.
\end{lemma}

{\rm Now we need some results about Frobenius number and \gbb that  follow from \cite{M-D}. Let $R=K[x_{0},\ldots,x_{n}]$ be the polynomial ring graded by the weights  $\deg x_0=a_0,\ldots ,\deg x_n=a_n$,
 $J\subset R$ a graded ideal and $B=R/J$.
Set $R'=K[x_1,\ldots,x_n]$ and  denote by $[[R']]$ the set of all monomials of $R'.$ Let $\varphi\colon  [[R']]\rightarrow \bn$ be the map defined by  $\varphi(M)=k_1 a_1+\ldots + k_n a_n$, for every monomial $M=x_1^{k_1}\ldots x_n^{k_n} \in [[R']]$.
We consider   the weighted degree reverse lexicographical order $\prec _{w}$ with  $x_0\prec _{w}\cdots \prec _{w}x_n$ and  $\deg x_i=a_i$ for all $0\leqslant i\leqslant n$. Let ${\rm in}(I)$ be the initial ideal  of $I$ for the order
$\prec _{w}$, that is the ideal generated by the leading terms of all the polynomials  in $I$.   Now we consider  two sets  
 $$\widetilde{\Ap(S,a_0)}\defeq \{ M \in [[R']] \mid  M \notin {{\rm in}(I)}\}$$
and
$$\widetilde{PF(S)}\defeq \{ M  \in \widetilde{\Ap(S,a_0)} \mid \forall i\ne 0,  \exists N_i\in [[R']],  \alpha_i>0 \text{\ such that\ }
 Mx_i-x_0^{\alpha_i}N_i \in I\}.$$

\begin{corollary}\label{c10} Assume that $\gcd(a_0,\ldots ,a_n)=1$. We have  
 
(i)  The restriction of $\varphi$ to $ \widetilde{\Ap(S,a_0)}$  is bijective and $\varphi (\widetilde{\Ap(S,a_0)})=\Ap(S,a_0)$.
In particular
$\card(\widetilde{(\Ap(S,a_0))})=a_0$ and $F(S)=\max\{ { \varphi(M) \mid   M\notin \In(I)}\}-{a_0}$.
 
(ii) The restriction of $\varphi$ to $\widetilde{PF(S)}$   is bijective and $\varphi (\widetilde{PF(S)})=PF(S)+a_0$, i.e. each  
element   $\omega \in PF(S) $ corresponds to exactly one monomial  $M_\omega \in \widetilde{PF(H)}$
such that $\varphi (M_\omega)-a_0=\omega.$

(iii) Maximality of   $ \widetilde{PF(S)}$: let $M\in [[R']]$,   if $x_iM\in \widetilde{\Ap(S,a)}$ for some   $i=1,\dots ,k+1$ then certainly $ M\not\in \widetilde{PF(S)}$.

(iv) Let $s\in {\rm Ap}(S,a_0)$, $M\in \widetilde{\Ap(S,a_0)}$ and $N\in [[R']]$ such that $s=\varphi(M)=\varphi(N)$. Then  $M\prec _{w} N$.
\end{corollary}

We denote by $\widetilde{\rm Frob(S) }$ the unique monomial in $\widetilde{ PF(S)}$ such that $\varphi (\widetilde{\rm Frob(S) })=F(S)+a_0$. The  following lemma    characterizes almost symmetric numerical semigroups in terms of monomials.
 
 \begin{lemma}\label{Nari-M} Suppose $t(S)>1$. Let $\widetilde{PF(S)}=\{M_1,\dots, M_{t(S)-1}, \widetilde{\rm Frob(S) }\}$ with $\varphi (M_1)<\dots<\varphi (M_{t(S)-1})$. Then $S$ is almost symmetric if and only if
 $M_i M_{t(S)-i}- x_0 \widetilde{\rm Frob(S)}\in I,$ for $i=1,\dots,t(S)-1$.
\end{lemma}

\begin{proof} Let $PF(S)=\{f_1,\dots, f_{t(S)-1}, F(S) \}$. Since   $\varphi (M_i )=a_0+f_i$, we have  $\varphi (M_i M_{t(S)-i})=2a_0+f_i+f_{t(S)-i}$ and
$\varphi (\widetilde{\rm Frob(S)})=a_0+F(S)$. So our claim follows by Lemma \ref{Nari}.
\end{proof}

\section{Numerical generalized almost arithmetic semigroups (AAG-semigroups)} 
 
Let  $a, k,h, c$ be positive integers and $d\in \fz^*$ such that $ha+kd>0$. Recall from \cite{M-D} that a {\it numerical  generalized almost  arithmetic  semigroup } (AAG-semigroup for short) is a semigroup minimally generated by relatively prime positive integers $a, ha+d, ha+2d, \ldots, ha+kd, c, $ that is its embedding dimension is $k+2$.  
 An interesting particular case   when $h=1, d>0$ called a {\it numerical almost   arithmetic  semigroup}. They were considered by D. P. Patil  \cite{P} from the algebraic point of view and by  J. L. Ram\'irez Alfons\'in  and  O. J. R\o dseth in \cite{RR}, \cite {ROD2} from combinatorial point of view.
  
 Let $R= K[x_{0},\ldots,x_{k}, x_{k+1}]$ be the polynomial ring  in $k+2$ variables over $K$  graded by the weights  $\deg x_i=ha+id$ for $i=0,\ldots,k,$  $\deg x_{k+1}=c$ and  $I$  the kernel of the homomorphism $\Phi\colon  R  \rightarrow K[S]$ of $K$-algebras defined by $\Phi(x_0)=t^{a}$, $\Phi(x_i)=t^{ha+id} $ for all  $i=1,\ldots, k$ and $\Phi(x_{k+1})=t^c.$ Let  $R'= K[x_1,\ldots,x_{k+1}]$.
 
\subsection{\gbb of AAG-semigroups}

\begin{remark}\label{rneg} {\rm If $d<0$ then we can assume $h\geq 2$. Namely, suppose  $d<0$ and $h=1$,   $S$ is generated by $a, a+d, a+2d, \ldots, a+kd, c $. If we set $a'= a+kd, d'=-d>0$ then $S$ is generated by $a', a'+d', a'+2d', \ldots, a'+kd', c$. }
 \end{remark}
 
 We recall the following results from \cite[Section 5]{M-D}. The following result extends \cite[Lemma 1.6.1]{P}. 
 
 \begin{lemma}\label{determinantal} For $1\leq i,j<k$, set 
 	$$\mathcal A\defeq \{x_ix_j- x_{0}^h x_{i+j}\mid  {\rm if\ }    i+j\leq k,\} \cup \{x_ix_j- x_{i+j-k}x_{k}\mid {\rm if\ }    i+j> k\}.$$
 	Then every binomial of $\mathcal A$ belongs to $I$. Note that  $\card(\mathcal A)=\frac{(k-1)k}{2}$. 
 \end{lemma} 
 
 \begin{corollary}\label{orderGB} Let consider the weighted monomial order $\prec _{w}$ such that $x_0\prec _{w}x_1\prec _{w}\ldots \prec _{w}x_k\prec _{w}x_{k+1}$.  The leading terms of $\mathcal A$ are part of the generators of $\In(I)$. Moreover     $\In(I)\cap [[R']]$ and $\widetilde{\Ap(S,a_0)}$ can be represented in the plane. Indeed, to any point $(y,z)\in \fn^2$ we associate the monomial 
 	$M(y,z)  \defeq L_ix_k^{\alpha}x_{k+1}^z,$ where $\alpha={\big\lfloor \frac{y}{k} \big\rfloor}$, $ i= y-k\alpha$ and  $L_i=1$   if $ i=0$  and $L_i=x_i $   if  $0<i<k.$  Conversely, any monomial $L_ix_k^{\alpha}x_{k+1}^z \in  [[R']]$  can be represented by the point $(y,z)\in \fn^2$, where $y=\alpha k+i.$
 \end{corollary}
 
 \begin{lemma}\label{equations-s-p} 
 	Let $s  \in \fn,p,r \in \fz $ such that 
 	$$ ra=sd-pc.$$
 	Let $s=\sigma k+l\rho,$ where $0\leq \rho <k$  and   $l=0$ if $\rho =0$,  $l=1$ if $\rho >0$. 
 	\item (1) We have 
 	$$ (r+h(\sigma +l))a  = \sigma   (ha+kd) +l(ha+\rho    d)-p c  .\ \ (*)$$
 	\item (2) Define   
 	$r'\defeq r+h(\sigma +l)$, we have \begin{itemize}\item  If $p,r'\geq 0$ then $L_lx_k^\sigma  -x_0 ^{r'}x_{k+1}^p\in I $.
 		\item  If $p<0,r'> 0$ then $L_lx_k^\sigma x_{k+1}^{-p} -x_0 ^{r'} \in I $. 
 		\item  If $p\geq 0,r'< 0$ then  $x_{k+1}^{p} -x_0 ^{-r'} L_l x_k^\sigma \in I$. 
 	\end{itemize}
 \end{lemma}
 
 Now we  define  numbers $s_i, p_i, r_i, q_i$ for $i\geq 2 $ such that $(s_i,p_i,r_i)$ is solution of the equation $sd-pc=ra$.
 We apply our algorithm for the case $n=3$ in \cite{Mo1} with numbers $a, d,c$.
 
 Let $s_0$ be the smallest natural number such that 
 $(s_0,0,r_0)$ is solution of the equation $sd-pc=ra$. Set $p_0=0$ and let $p_1$ be the smallest natural number such that
 $(s_1,p_1,r_1)$ is solution of the equation $sd-pc=ra$, where  $0\leq s_1<s_0$. 
 In our situation by using \cite[Proposition 1.3]{Mo1} we can assume $\gcd(a,d)=1$, so $s_0=a, p_1=1$.  Let consider the Euclid's algorithm with negative rest:  
 $$ \left \{
 \begin{array}{r c l}
 	s_{0} & = & q_2s_1-s_2 \\
 	s_1 & = & q_3s_2-s_3 \\
 	\ldots  & = & \ldots \\
 	s_{m-1} & = & q_{m+1}s_{m} \\
 	s_{m+1} & = & 0 
 \end{array}
 \right.       $$
 where  $ q_i\ge 2,\ s_i\ge 0$ for all  $i=2,\ldots, m+1.$ For $i=1, \ldots, m$, let define  $ p_{i+1}, r_{i+1}$   by
 $$  p_{i+1}=p_iq_{i+1}-p_{i-1}~,  r_{i+1}=r_iq_{i+1}-r_{i-1}.$$
 It is proved in \cite{Mo1},  \cite{Mo2} that for $i=0,\ldots, m,$  a triple number $(s_i,p_i,r_i)$ is solution of the equation $sd-pc=ra$,
 $$s_i p_{i+1}- s_{i+1}p_i=a, s_{i+1}r_i-s_i r_{i+1}=c, p_{i+1}r_i-p_i r_{i+1}=d,$$
 and the sequence $s_i$ is decreasing, while the sequence $p_i$ is increasing. Moreover if $d>0$ then the sequence $ r_i$ is decreasing.
 
 Let $s_i= \sigma _i k+l_i \rho _i,$ where $0\leq \rho _i<k$ and  $l_i=0 $ if   $\rho _i=0$, $l_i=1 $ if   $\rho _i>0 $. We also set     
 $ r'_i= r_i+h(\sigma _i+l_i)$. If $s_{ \mu+1}\not=0$ let $ s_\mu-s_{ \mu+1}= \widetilde \sigma k+ \widetilde l \widetilde \rho$, with $0\leq \widetilde \rho<k,$  and $\widetilde l=0$ if $\widetilde \rho=0$, $\widetilde l=1$ if $\widetilde \rho>0$. We define $\widetilde  r\defeq   r_\mu-r_{ \mu+1} +h(\widetilde \sigma +\widetilde l)$.  
 
 \begin{lemma}\label{widetildehorizontal0} (\cite[Lemma 5.4]{M-D}) For $i=0,\ldots, m+1$ we have 
 	$$L_{\widetilde \rho_{i}} x_k^{\widetilde \sigma _{i}}x_{k+1}^{p_{i+1}-p_i}-x_0^{\widetilde r _{i}}\in I,$$
 	where  $\widetilde r _{i}>h$ if    ${\widetilde \rho_{i}}>0$. Moreover for $i=0,\ldots, m,$ we have  $r'_i> r' _{i+1}.$  
 \end{lemma}
 
 Note that we always  have $\widetilde r _{i}\geq 2$ since we assume that the embedding dimension of $S$ is $k+2$.
 
 \begin{definition} {\rm Let $ \mu$ be the unique integer such that   $ r'_{\mu }> 0\geq r'_{\mu+1}$. We collect the {\it sequences} in the following table \begin{tabular}{lll|l}
 		s & p & r & r'\\
 		\hline 
 		$\cdots$& $\cdots$& $\cdots$&$\cdots$\\
 		$s_\mu $ & $p_\mu $ & $r_\mu $& $r'_\mu $\\
 		$s_{\mu +1}$ & $p_{\mu +1}$ & $r_{\mu +1}$ & $r'_{\mu +1}$ \\
 		$\cdots$& $\cdots$& $\cdots$&$\cdots$\\
 	\end{tabular}}
 \end{definition}
 
 As in \cite{M-D},  from now on we  suppose that either  $r'_\mu\geq h$ or $ \rho_\mu =0$. Note that by the definition of $\mu $, if $h=1$ then  $r'_\mu\geq h$.
 
 \begin{definition}{\rm  With the above notations and the assumption $r'_\mu\geq h$ or $ \rho_\mu =0$, we define some {\it sets}:  
 		\item(1) If  $ \rho_\mu =0$ then we set  $\displaystyle \mathcal B\defeq \{  x_k^{\sigma_\mu} -x_0^{r'_\mu }x_{k+1}^{p_\mu}\}.$ 
 		If  $ \rho_\mu \not=0$ then we set $$\displaystyle \mathcal B\defeq \{ x_{\rho_\mu } x_k^{ \sigma_\mu } -x_0^{r'_\mu }x_{k+1}^{p_\mu} , x_{\rho_\mu +j} x_k^{\sigma_\mu } -x_0^{r'_\mu-h} x_{j} x_{k+1}^{p_\mu} \mid 1\leq j\leq k-\rho_\mu \}.$$
 		
 		\item(2)   If $\widetilde \rho>0$ then we set 
 		$$\displaystyle \mathcal C\defeq \{ x_{\widetilde \rho}x_k^{\widetilde \sigma} x_{k+1}^{p_{ \mu+1}-p_\mu}-x_0^{\widetilde  r},x_{j+\widetilde \rho} x_k^{\widetilde \sigma} x_{k+1}^{p_{ \mu+1}-p_\mu}-x_0^{\widetilde  r-h}x_j \mid  1\leq j\leq  k-\widetilde \rho \},$$
 		if  $\widetilde \rho=0 $ then we set 
 		$\displaystyle \mathcal C\defeq \{  x_k^{\widetilde \sigma} x_{k+1}^{p_{ \mu+1}-p_\mu}-x_0^{\widetilde  r}\}, $ and if $s_{ \mu+1} =0$ then we set $\mathcal C\defeq \emptyset$.
 		
 		\item(3) $\displaystyle \mathcal D\defeq\{x_{k+1}^{p_{ \mu+1}}-x_0^{-r'_{ \mu+1}}L_{\rho_{ \mu+1}}x_k^{\sigma_{ \mu+1}} \}$. 
 		
 		By our assumptions the embedding dimension of the semigroup $S$ is $k+2,$ so $s_\mu >k$, $p_{ \mu+1}>1$ which implies $\mu >0$, and  if $r'_{ \mu+1}=0$  then $s_{ \mu+1}>k$, or reciprocally if $s_{ \mu+1}\leq k$ then $r'_{ \mu+1}<0$.}
 \end{definition}
 
 \medskip
 The next theorem allows us to compute effectively a \gbb of the ideal semigroup $I$, it precises and   extends the main theorem of \cite {P} where the case $h=1$ is considered.
 
 \begin{theorem} \label{AAG-1} Let $S$ be a AAG-semigroup. Suppose that either  $r'_\mu\geq h$ or $ \rho_\mu =0$. 
 	\item (i) We have   \[\begin{aligned}\widetilde{\Ap(S,a)}&=\left\{L_ix_k^{\alpha}x_{k+1}^z \mid 0\leq \alpha k+i <s_{\mu }-s_{\mu+1}, 0\leq z<p_{\mu+1 }\right\}\cup\\
 		& \left\{L_ix_k^{\alpha}x_{k+1}^z \mid
 		s_{\mu }-s_{\mu+1}\leq \alpha k+i <s_{\mu },0\leq z<p_{\mu+1 }-p_{\mu}\right\}.\end{aligned}\]
 	
 	\item(ii) 
 	$\mathcal G\defeq \mathcal A\cup \mathcal B\cup \mathcal C\cup \mathcal D$ is a \gbb of $I$.
 \end{theorem} 

 \begin{corollary} \label{c-AAG-1} Suppose that either  $r'_\mu\geq h$ or $ \rho_\mu =0$. With the  notations of Corollary \ref{orderGB}.
 Let $$ \mathcal U=\{M(y,z)\mid 0\leq y< s_{\mu }-s_{\mu+1}, z\geq p_{\mu+1 }\}, $$
$$\mathcal V=\{M(y,z)\mid  y\geq  s_{\mu }-s_{\mu+1}, z\geq p_{\mu+1 }-p_{\mu}\}, $$
$$\mathcal W=\{M(y,z)\mid  y\geq  s_{\mu }, 0\leq z< p_{\mu+1 }-p_{\mu}\}.$$ Then $\In(I)=\mathcal U\cup \mathcal V\cup \mathcal W.$ 
 \end{corollary} 

 From now on we will use  the  notations introduced in Corollaries \ref{orderGB}, \ref{c-AAG-1}.
 \begin{lemma}\label{AAG-2} Suppose that either  $r'_\mu\geq h$ or $ \rho_\mu =0$. 
 	 Let $M\in \In(I)$, if either 
 	\item $r'_{\mu+1 }<0$ and $M\in \mathcal U$, or
  \item $M\in \mathcal V$, or 	
\item $ r'_\mu>h $ or $ \rho_\mu =0$ and  $M\in \mathcal W$, \\ 
 then there exists $\alpha>0 $ and a monomial $N$ such that $M-x_0^\alpha N\in I$. 
 \end{lemma}
 
 \medskip
 The following lemma allows us to  compare $\widetilde r$ and $r'_\mu - r' _{\mu +1}$.
 
 \begin{lemma}\label{widetildehorizontal}  \begin{enumerate}
 		\item If $\rho_\mu=\rho_{\mu+1}=0$ then   $\widetilde r  =r'_\mu- r' _{\mu+1}.$
 		\item If $\rho_\mu=0,\rho_{\mu+1}>0$ then   $\widetilde r -h =r'_\mu- r' _{\mu+1}$.
 		\item If $\rho_\mu>0,\rho_{\mu+1}=0$ then   $\widetilde r  =r'_\mu- r' _{\mu+1}.$
 		\item If $\rho_\mu>0,\rho_{\mu+1}>0$ and $0<\rho _\mu<\rho _{\mu+1}$ then  $\widetilde r=r'_\mu- r' _{\mu+1}.$
 		\item If $\rho_\mu>0,\rho_{\mu+1}>0$ and  $ 0<\rho _\mu=\rho _{\mu+1}$ then  $\widetilde r  =r'_\mu- r' _{\mu+1}.$
 		\item If $\rho_\mu>0,\rho_{\mu+1}>0$ and  $ \rho _\mu>\rho _{\mu+1}>0$ then $\widetilde r-h =r'_\mu- r' _{\mu+1}$.
 	\end{enumerate}
 \end{lemma}
 
 \begin{proof} 
 	Recall that 
 	$\widetilde r :=r_\mu-r_{\mu+1}+h( \widetilde \sigma +\widetilde l )$, 
 	$r'_\mu- r' _{\mu+1}=r_\mu-r_{\mu+1}+h(  \sigma _{\mu}-\sigma _{\mu+1}+l_\mu-l_{\mu+1})$. We have six cases.
 	\begin{enumerate}
 		\item If $\rho_\mu=\rho_{\mu+1}=0$  then  $\rho _\mu=\rho _{\mu+1}=\widetilde \rho=0, \widetilde l=0,\widetilde \sigma =\sigma _\mu-\sigma _{\mu+1}$. Hence  $\widetilde r  =r'_\mu- r' _{\mu+1}.$
 		\item If $\rho_\mu=0,\rho_{\mu+1}>0$ then  $\rho _\mu=0,\rho _{\mu+1}>0,\widetilde \rho=k-\rho _{\mu+1}>0,\widetilde l=1,$ and  $\widetilde \sigma=\sigma _\mu-\sigma _{\mu+1}-1$. Hence $\widetilde r -h =r'_\mu- r' _{\mu+1}$.
 		\item If $\rho_\mu>0,\rho_{\mu+1}=0$  then  $\rho _\mu=\widetilde \rho, \widetilde l=1,\widetilde \sigma =\sigma _\mu-\sigma _{\mu+1}$. Hence  $\widetilde r  =r'_\mu- r' _{\mu+1}.$
 		\item If $\rho_\mu>0,\rho_{\mu+1}>0$ and $0<\rho _\mu<\rho _{\mu+1}$ then $  \widetilde \rho=k+\rho_{\mu}-\rho _{\mu+1}>0, \widetilde l=1,$ and $\widetilde \sigma=\sigma _\mu-\sigma _{\mu+1}-1\geq 0$. Hence $\widetilde r=r'_\mu- r' _{\mu+1}.$
 		\item If $\rho_\mu>0,\rho_{\mu+1}>0$ and  $ 0<\rho _\mu=\rho _{\mu+1}$ then $ \widetilde \rho=0, \widetilde l=0, \widetilde \sigma=\sigma _\mu-\sigma _{\mu+1}\geq 0$.  Hence $\widetilde r  =r'_\mu- r' _{\mu+1}.$
 		\item If $\rho_\mu>0,\rho_{\mu+1}>0$ and  $ \rho _\mu>\rho _{\mu+1}>0$ then $ \widetilde \rho=\rho _\mu-\rho _{\mu+1}>0, \widetilde l=1,$ and $ \widetilde \sigma=\sigma _\mu-\sigma _{\mu+1}\geq 0.$ Hence $\widetilde r-h =r'_\mu- r' _{\mu+1}$.
 	\end{enumerate}
 \end{proof} 
 
 For a description of the Apéry set and a formula for the Frobenius number in terms of the sequences $s_i,p_i,r_i$ (see \cite[section 5]{M-D}).
 
\subsection{Pseudo Frobenius set of AAG-semigroups}

Now we consider pseudo Frobenius set of AAG-semigroups.
 When $h=1$, the case $S$  symmetric was studied  in \cite{RR} and the case $S$ is pseudo symmetric was studied in \cite{GRR}. In our work we will describe the pseudo Frobenius set and characterize when $S$ is  almost symmetric in general for  $h \geq 1 $. In this paper we restrict to the hypothesis  $r'_\mu\geq h$ or $ \rho_\mu =0$. Note that $r'_\mu\geq h$ is satisfied if $h=1$.  
 Let recall the following lemma.
 
 \begin{lemma}\label{lPSF} (\cite[Corollary 5.9]{M-D})
 	With the above notations, suppose that either  $r'_\mu\geq h$ or $ \rho_\mu =0$. 
 	Let  $\widetilde{PF(S)}_1 $ be  the set of  monomials in $\widetilde{PF(S)}$ such that the power of $x_{k+1}$ is  $p_{\mu+1 }-1$ and let $\widetilde{PF(S)}_2  $ be  the set of  monomials in $\widetilde{PF(S)}$ such that the power of $x_{k+1}$ is  $p_{\mu+1 }-p_{\mu} -1$. We have
 	$\widetilde{PF(S)}=\widetilde{PF(S)}_1\cup \widetilde{PF(S)}_2.$ 
 \end{lemma}
 
 In the proof of Theorem \ref{tPSF}, we will need the following remark.
  \begin{remark}\label{rPF}{\rm (a) Let $i\in \{1,\dots ,k\}, j\in \{1,\dots ,k+1\}$, $x_iM\in \widetilde{\Ap(S,a)}$.  If  $i+j\leq k$ then  $x_jx_iM-x_0^hx_{i+j}M \in I$ since $x_ix_j- x_0^h x_{i+j}\in I $. So in order to check if $x_iM$  belongs to $\widetilde{PF(S)}$  we need only to consider $j$ such that $i+j> k$.
 		
 		\item(b) Since  
 		$ s_{\mu }-s_{\mu+1} = (\sigma_{\mu }-\sigma_{\mu+1})k + \rho_{\mu }-\rho_{\mu+1}$, so if $\rho_{\mu+1 }\leq \rho_{\mu }$ then we have 
 		$\widetilde{\sigma  }=\sigma_{\mu }-\sigma_{\mu+1}, \widetilde{\rho }=\rho_{\mu }-\rho_{\mu+1}$ and if 
 		$\rho_{\mu+1 }> \rho_{\mu }$ then we have 
 		$\widetilde{\sigma  }=\sigma_{\mu }-\sigma_{\mu+1}-1, \widetilde{\rho }=k+\rho_{\mu }-\rho_{\mu+1}$.}
 		 \end{remark}
 	 
 In the next theorem we describe the pseudo Frobenius set, it will be very useful to study symmetric and almost symmetric AAG-semigroups. 
 
 \begin{theorem}\label{tPSF} With the above notations, assume that either  $r'_\mu\geq h$ or $ \rho_\mu =0$.
 	We have
 	\begin{enumerate}\item  Suppose $r'_ {\mu +1 }=0$.
 		\begin{enumerate}
 			\item If $ \rho_{\mu+1 }=0 $ then $\widetilde{PF(S)}_1=\emptyset .$
 			\item If $\rho_{\mu+1 }>0,\widetilde{\rho }=0   $  then 
 			$\displaystyle \widetilde{PF(S)}_1=\{x_i x_k ^{\widetilde{\sigma  }-1 } x_{k+1}^{p_{\mu+1 } -1}, i=1,\dots ,k-\rho_{\mu+1 }  \}.$
 			\item If $\rho_{\mu+1 }>0, \widetilde{\rho }=1, \widetilde{\sigma }=0$ then 
 			$\widetilde{PF(S)}_1=\emptyset .$ 
 			\item If $\rho_{\mu+1 }>0,\widetilde{\rho }=1, \widetilde{\sigma }>0$  then \\
 			$\displaystyle \widetilde{PF(S)}_1=\{x_i x_k ^{\widetilde{\sigma  }-1 } x_{k+1}^{p_{\mu+1 } -1},i=1,\dots ,k-\rho_{\mu+1 }  \}.$
 			\item If $\rho_{\mu+1 }>0, \widetilde{\rho }>1$ then \\
 			$\displaystyle \widetilde{PF(S)}_1=\{x_i x_k ^{\widetilde{\sigma  } } x_{k+1}^{p_{\mu+1 } -1}, i=1,\dots ,\min{\widetilde{\rho }-1,k-\rho _{\mu+1 }  } \}.$
 		\end{enumerate} 
 		\item Suppose $r'_ {\mu +1 }<0$.
 		\begin{enumerate}
 			\item If $\widetilde{\rho }=0$ then $\displaystyle \widetilde{PF(S)}_1=\{x_i x_k ^{\widetilde{\sigma  }-1 } x_{k+1}^{p_{\mu+1 } -1}, i=1,\dots ,k-1  \}.$
 			\item If $\widetilde{\rho }=1, \widetilde{\sigma }=0 $ then 
 			$\displaystyle \widetilde{PF(S)}_1=\{ x_{k+1}^{p_{\mu+1 } -1} \}.$
 			\item If $\widetilde{\rho }=1, \widetilde{\sigma }>0 $ then $\displaystyle \widetilde{PF(S)}_1=\{x_i x_k ^{\widetilde{\sigma  }-1 } x_{k+1}^{p_{\mu+1 } -1}, i=1,\dots ,k  \}.$
 			\item If $\widetilde{\rho }>1$ then $\displaystyle \widetilde{PF(S)}_1=\{x_i x_k ^{\widetilde{\sigma  } } x_{k+1}^{p_{\mu+1 } -1}, i=1,\dots ,\widetilde{\rho }-1 \}.$
 		\end{enumerate} 
 		\item If  $s_{\mu +1 }=0 $ then $\widetilde{PF(S)}_2=\emptyset $.
 		\item  Suppose $\rho_{\mu }=0 $.
 		\begin{enumerate}[label=(\roman*)]
 			\item If $s_{\mu+1} \geq  k-1$  then $\displaystyle \widetilde{PF(S)}_2=\{x_i x_k ^{\sigma_{\mu }-1} x_{k+1}^{p_{\mu+1 }-p_{\mu} -1}, i=1,\dots ,k-1  \}.$
 			\item If $s_{\mu+1} < k-1$  then $\displaystyle \widetilde{PF(S)}_2=\{x_i x_k ^{\sigma_{\mu }-1} x_{k+1}^{p_{\mu+1 }-p_{\mu} -1}, i=\widetilde{\rho },\dots ,k-1  \}.$
 		\end{enumerate} 
 		\item Suppose $\rho_{\mu }=1, r'_{\mu }>h$.
 		\begin{enumerate}[label=(\roman*)]
 			\item If  $ s_{\mu+1} \geq  k$  then $\displaystyle \widetilde{PF(S)}_2=\{x_i x_k ^{\sigma_{\mu }-1} x_{k+1}^{p_{\mu+1 }-p_{\mu} -1}, i=1,\dots ,k\}.$
 			\item If  $ 1<s_{\mu+1} <  k$ then $\displaystyle \widetilde{PF(S)}_2=\{x_i x_k ^{\sigma_{\mu }-1} x_{k+1}^{p_{\mu+1 }-p_{\mu} -1}, i=\widetilde{\rho },\dots ,k  \}.$
 			\item If  $ s_{\mu+1} =1$  then $\displaystyle \widetilde{PF(S)}_2=\{ x_k ^{\sigma_{\mu }} x_{k+1}^{p_{\mu+1 }-p_{\mu} -1} \}.$
 		\end{enumerate}
 		\item Suppose $\rho_{\mu }=1, r'_{\mu }=h$.
 		\begin{enumerate}[label=(\roman*)]
 			\item If $s_{\mu }- s_{\mu+1}=1$  then $\displaystyle \widetilde{PF(S)}_2=\{x_i x_k ^{\sigma_{\mu }-1} x_{k+1}^{p_{\mu+1 }-p_{\mu} -1}, i=1,\dots ,k\}.$
 			\item If  $1<s_{\mu }- s_{\mu+1}\leq s_{\mu }-k$  then $\displaystyle \widetilde{PF(S)}_2=\{x_1 x_k ^{\sigma_{\mu }-1} x_{k+1}^{p_{\mu+1 }-p_{\mu} -1}\}.$
 			\item If  $s_{\mu }-k<s_{\mu }- s_{\mu+1}$  then  $\widetilde{PF(S)}_2=\emptyset.$
 		\end{enumerate}
 		\item Suppose $\rho_{\mu }>1$.
 		\begin{enumerate}[label=(\roman*)]
 			\item If $ s_{\mu+1}\geq  \rho _{\mu }-1$  then $\displaystyle \widetilde{PF(S)}_2=\{x_i x_k ^{\sigma_{\mu }} x_{k+1}^{p_{\mu+1 }-p_{\mu} -1}, i=1,\dots ,\rho_{\mu }-1\}.$
 			\item If $ s_{\mu+1}<  \rho _{\mu }-1$  then $\widetilde{\rho }=\rho _{\mu }-s_{\mu+1}$ and \\  $\displaystyle \widetilde{PF(S)}_2=\{x_i x_k ^{\sigma_{\mu }} x_{k+1}^{p_{\mu+1 }-p_{\mu} -1}, i=\widetilde{\rho },\dots ,\rho_{\mu }-1\}.$
 		\end{enumerate}
 	\end{enumerate}
 \end{theorem} 
 
 \begin{proof}
 	We have  to consider all possible cases: \\ 
 	(1) Prove claims (1) and (2) for studying $ \widetilde{PF(S)}_1$.
 	\begin{enumerate}[label=\alph*)]
	\item  Suppose $s_ {\mu}-s_ {\mu+1}=1$. Then we can see that   $\widetilde{PF(S)}_1 \subset \{ x_{k+1}^{p_{\mu+1 } -1}\}$ and $x_{k+1}(x_{k+1}^{p_{\mu+1 }-1})-x_0^{-r' _{\mu+1 }}x_{\rho _{\mu+1 }}x_{k}^{\sigma _{\mu+1 }}\in I$.\\ 
	Suppose $r' _{\mu+1 }=0$, then we have 
 		$\varphi (x_{\rho _{\mu+1 }}x_{k}^{\sigma _{\mu+1 }})=s_{\mu+1}<s_{\mu}$. It implies that    
 		$x_{\rho _{\mu+1 }}x_{k}^{\sigma _{\mu+1 }}\in \widetilde{\Ap(S,a)}.$ Thus $\widetilde{PF(S)}_1 =\emptyset$.\\
 		Suppose $r' _{\mu+1 }<0$.  For  $j=1,\dots ,k$, 
 		  we have  $x_{j}x_{k+1}^{p_{\mu+1 } -1}\in \mathcal V$, so by Lemma \ref{AAG-2} there exists $\alpha  >0$ and a monomial $N$ such that $x_{j}M_i- x_0^\alpha N$. Thus $\widetilde{PF(S)}_1 = \{ x_{k+1}^{p_{\mu+1 } -1}\}$.
 		\item  \label{item1}Suppose $ s_ {\mu}-s_ {\mu+1}>1$. Then we have 
$\widetilde{PF(S)}_1 \subset \{x_i x_k ^{\gamma   } x_{k+1}^{p_{\mu+1 } -1}, i=1,\dots ,A \} $, where  
 
If $\widetilde{\rho }=0$ then $\gamma  = \widetilde{\sigma  }-1$, and $A=k-1$, 

 		If $\widetilde{\rho }=1, \widetilde{\sigma }>0$ then $\gamma  = \widetilde{\sigma  }-1$, and $A=k$, 
 		
 		 If $ \widetilde{\rho }>1$ then   $\gamma  = \widetilde{\sigma  },$ and $A=\widetilde{\rho }-1$.\\ 		  		
 		Let $M=x_k ^{\gamma } x_{k+1}^{p_{\mu+1 } -1}, j=1,\dots ,k+1.$ By Remark \ref{rPF}, we have to study only the monomial $x_{j}x_i M$ for $i+j>k$. If 
 		$j\leq k$ then we have $x_{j}x_i M\in \mathcal V$. So by Lemma \ref{AAG-2} there exists $\alpha  >0$ and a monomial $N$ such that $x_{j}x_i M- x_0^\alpha N$.
 		If 
 		$j=k+1$ then we have $x_{k+1}x_i M- x_0^{-r' _{\mu+1 }} x_i L_{\rho _{\mu+1 }}  x_{k}^{\sigma _{\mu+1 }}x_k ^{\gamma }\in I$.\\
It then follows that if $r' _{\mu+1 }>0$ then $\widetilde{PF(S)}_1 = \{x_i x_k ^{\gamma   } x_{k+1}^{p_{\mu+1 } -1}, i=1,\dots ,A \} $. Now we consider the case $r' _{\mu+1 }=0$. First note that $s_\mu = (\widetilde{\sigma  }+ \sigma _{\mu+1 } )k+ \rho_{\mu+1 } +\widetilde{\rho }$.
\begin{itemize}
\item If $\rho_{\mu+1 }=0$ then   $x_{k+1}x_i M- x_i  x_{k}^{\sigma _{\mu+1 }}x_k ^{\gamma }\in I$, and we have several cases:\\
 If  $\gamma  = \widetilde{\sigma  }-1, $ then $\varphi (x_i  x_{k}^{\sigma _{\mu+1 }}x_k ^{\gamma })=i+(\sigma _{\mu+1 }+\widetilde{\sigma  }-1)k=s_\mu +i-k-\widetilde{\rho }<s_\mu$. Thus  $x_i  x_{k}^{\sigma _{\mu+1 }}x_k ^{\gamma }\in \widetilde{\Ap(S,a)}$. Hence $\widetilde{PF(S)}_1=\emptyset$.\\
 If  $\gamma  =  \widetilde{\sigma  } $ then $ i+(\widetilde{\sigma  }+\sigma _{\mu+1 })k=s_\mu +i-\widetilde{\rho }< s_\mu $ since  $i<\widetilde{\rho }$, hence $\widetilde{PF(S)}_1=\emptyset$.

\item Now suppose $\rho_{\mu+1 }>0$. We have $x_{k+1}x_i M- x_0^{-r' _{\mu+1 }} x_i x_{\rho _{\mu+1 }}  x_{k}^{\sigma _{\mu+1 }}x_k ^{\gamma }\in I.$
 If $i+ \rho _{\mu+1 }\leq k$ then  $x_ix_{\rho _{\mu+1 }} -x_0^h x_{i+\rho _{\mu+1 }} \in I$, so $x_i M\in \widetilde{PF(S)}$. 
If $i+ \rho _{\mu+1 }> k$ then  $x_{k+1}x_i M-x_{i+\rho _{\mu+1 }-k} 
 x_{k}^{\sigma _{\mu+1 }+\gamma +1}\in I$.  We have to discuss several cases. \\
If  $\gamma  = \widetilde{\sigma  }-1, $ then $i+\rho _{\mu+1 }-k+(\widetilde{\sigma  }+\sigma _{\mu+1 })k= s_\mu +i-k-\widetilde{\rho }<s_\mu ,$ which implies that $x_{i+\rho _{\mu+1 }-k} x_k ^{\widetilde{\sigma  }+\sigma _{\mu+1 }} \in \widetilde{\Ap(S,a)}.$ Therefore $x_i M\not \in \widetilde{PF(S)}$. Hence 
 		$$\widetilde{PF(S)}_1 =\{x_i x_k ^{\widetilde{\sigma  }-1 } x_{k+1}^{p_{\mu+1 } -1}, i=1,\dots ,k-\rho_{\mu+1 }  \}.$$
 		If   $\gamma  =  \widetilde{\sigma  } $ then $i+\rho _{\mu+1 }-k+(\widetilde{\sigma  }+\sigma _{\mu+1 })k+k=s_\mu +i-\widetilde{\rho }< s_\mu $,
 		which implies that $x_{i+\rho _{\mu+1 }-k} x_k ^{\widetilde{\sigma  }+\sigma _{\mu+1 }+1} \in \widetilde{\Ap(S,a)}.$ Thus $x_i M\not \in \widetilde{PF(S)}$. Hence 
 		$$\widetilde{PF(S)}_1 = \{x_i x_k ^{\widetilde{\sigma  } } x_{k+1}^{p_{\mu+1 } -1}, i=1,\dots ,\min{\widetilde{\rho }-1,k-\rho _{\mu+1 }  }\}.$$
 \end{itemize}	
 \end{enumerate} 
 
 \noindent (2) Prove claims (3), (4), (5), (6) and (7) for studying   $\widetilde{PF(S)}_2$.
	\begin{enumerate}[label=\alph*)]
 \item  If $s_{\mu+1 }=0$ then the set $ \widetilde{\Ap(S,a)}   $ is represented by a rectangle, there is no element in $ \widetilde{PF(S)}  $ with power of $x_{k+1}$ equal to $p_{\mu+1 }-p_{\mu} -1.$ Hence $ \widetilde{PF(S)}_2 =\emptyset. $

\item  Suppose $s_{\mu+1 }>0$. We have $$\widetilde{PF(S)}_2  \subset \{x_i x_k ^{\gamma } x_{k+1}^{p_{\mu+1 }-p_{\mu} -1}, i=1,\dots ,A  \},$$ 
 	where 
 	
If $\rho_{\mu }=0$ then $\gamma =\sigma_{\mu }-1, A=k-1$.

If $\rho_{\mu }=1$ then $\gamma =\sigma_{\mu }-1, A=k$.

If $\rho_{\mu }>1$ then $\gamma =\sigma_{\mu }, A=\rho_\mu  -1$.\\
(A) Let $j=1,\dots ,k, i=1,\dots ,A, M=x_k ^{\gamma } x_{k+1}^{p_{\mu+1 }-p_{\mu} -1}$, we study the monomial $x_jx_iM$. By Remark \ref{rPF}, we can assume that  $i+j>k$, hence $x_jx_iM- x_{i+j-k}x_k ^{\gamma +1} x_{k+1}^{p_{\mu+1 }-p_{\mu} -1}\in I$. We have to consider three cases:
\begin{enumerate}[label=(\Roman*)]
\item If $\rho_{\mu }=0$ then $\gamma+1 =\sigma_{\mu }$,  $x_{i+j-k}x_k ^{\sigma_{\mu }}x_{k+1}^{p_{\mu+1 }-p_{\mu} -1}\in \mathcal W.$ So by Lemma \ref{AAG-2} there exists $\alpha  >0$ and a monomial $N$ such that $x_{j}x_iM- x_0^\alpha N$.
\item If $\rho_{\mu }=1$ then $\gamma+1 =\sigma_{\mu }$. If  $i+j-k=1$ then $x_{1}x_k ^{\sigma_{\mu }}-x_0^{r'_{\mu }}x_{k+1}^{p_{\mu}}\in I$, so we can assume $i+j-k\geq 2$.  
By using the \gbb of $I$ we have $x_{i+j-k}x_k ^{\sigma_{\mu }}x_{k+1}^{p_{\mu+1 }-p_{\mu} -1}-x_0^{r'_{\mu }-h}x_{i+j-k-1}x_{k+1}^{p_{\mu+1 } -1}\in I$. If  $r'_{\mu }>h$ we have  $x_{j}x_iM-x_0^{r'_{\mu }-h}x_{i+j-k-1}x_{k+1}^{p_{\mu+1 } -1}\in I$. So we suppose $r'_{\mu }=h$, hence  
$x_{j}x_iM-x_{i+j-k-1}x_{k+1}^{p_{\mu+1 } -1}\in I$. Let consider two cases. 
 \begin{enumerate}[label=(\roman*)]
 	\item Suppose  $s_{\mu }-s_{\mu+1}=1$.  Since $i+j-k-1\geq 1 $ we have $x_{i+j-k-1}x_{k+1}^{p_{\mu+1 } -1}\in \mathcal V$,  so by Lemma \ref{AAG-2}, for $j=1,\ldots,k$ there exists $\alpha  >0$ and a monomial $N$ such that $x_{j}x_iM- x_0^\alpha N$.
 
 \item  Suppose $s_{\mu }-s_{\mu+1}>1$. For  $i>1$, we set  $j=k+2-i$,  so 
 		$x_jx_iM-x_{1}  x_{k+1}^{p_{\mu+1 }-1} \in I$, but $x_{1}  x_{k+1}^{p_{\mu+1 }-1}\in \widetilde{\Ap(S,a) }$. Consequently for $i>1$  we have $x_iM\notin \widetilde{PF(S)}_2 $, that is   $\widetilde{PF(S)}_2  \subset \{x_1 x_k ^{\sigma_{\mu }-1} x_{k+1}^{p_{\mu+1 }-p_{\mu} -1} \}.$
 		\end{enumerate} 
\item If $\rho_{\mu }>1$ then $\gamma =\sigma_{\mu }$. We have $x_jx_iM- x_{i+j-k}x_k ^{\sigma_{\mu } +1} x_{k+1}^{p_{\mu+1 }-p_{\mu} -1}\in I$ and 
$x_k ^{\sigma_{\mu }+1} - x_0^{r'_\mu -h}x_{k-\rho _{\mu }} x_{k+1}^{p_{\mu }} \in I$. Hence 
$x_jx_iM- x_0^{r'_\mu -h}x_{i+j-k}x_{k-\rho _{\mu }} x_{k+1}^{p_{\mu }}x_{k+1}^{p_{\mu+1 }-p_{\mu} -1}\in I$. Since $i<\rho _{\mu }$ we have $i+j-k+k-\rho _{\mu }<j\leq k$, and  $x_{i+j-k}x_{k-\rho _{\mu }} -x_0^h x_{i+j-\rho _{\mu }}$. So for any $j=1,\ldots,k, i=1,\ldots,\rho _{\mu }-1$ we have 
$x_jx_iM- x_0^{r'_\mu }x_{i+j-\rho _{\mu }} x_{k+1}^{p_{\mu }}x_{k+1}^{p_{\mu+1 }-p_{\mu} -1}\in I$.
\end{enumerate} 
(B) Now let study the monomial $x_{k+1}x_iM=x_i x_k ^{\sigma_{\mu }-1} x_{k+1}^{p_{\mu+1 }-p_{\mu}}.$ We have again three cases:
\begin{enumerate}[label=(\Roman*)]
 		\item
Suppose $\rho_{\mu }=0$. We have $x_i x_k ^{\sigma_{\mu }-1} x_{k+1}^{p_{\mu+1 }-p_{\mu}}\in \mathcal V$ if and only if $i+(\sigma_{\mu }-1)k\geq s_{\mu}-s_{\mu+1},$ that is $i\geq k-s_{\mu+1}$. We have two cases.
 	\begin{enumerate}[label=(\roman*)]
 		\item If $ s_{\mu+1}\geq k-1$ then $i\geq k-s_{\mu+1}$ for $i=1,\ldots,k-1$, so by  Lemma \ref{AAG-2} and the case (A) we have $\widetilde{PF(S)}_2  =\{x_i x_k ^{\sigma_{\mu }-1} x_{k+1}^{p_{\mu+1 }-p_{\mu} -1}, i=1,\dots ,k-1  \}.$ Hence $$\widetilde{PF(S)}_2  =\{x_i x_k ^{\sigma_{\mu }-1} x_{k+1}^{p_{\mu+1 }-p_{\mu} -1}, i=1,\dots ,k-1  \}.$$
 		\item If $ s_{\mu+1}< k-1$ then  $s_{\mu }-s_{\mu+1}= (\sigma_{\mu }-1)k+(k-s_{\mu+1})$ so $\widetilde{\rho }=k-s_{\mu+1}.$
 		Hence by  Lemma \ref{AAG-2} and the case (A) we have 
 		$$\widetilde{PF(S)}_2  =\{x_i x_k ^{\sigma_{\mu }-1} x_{k+1}^{p_{\mu+1 }-p_{\mu} -1}, i=\widetilde{\rho },\dots ,k-1  \}.$$
 	\end{enumerate} 
\item Suppose $\rho_{\mu }=1$. We have $x_i x_k ^{\sigma_{\mu }-1} x_{k+1}^{p_{\mu+1 }-p_{\mu}}\in \mathcal V$ if and only if $i+(\sigma_{\mu }-1)k\geq s_{\mu}-s_{\mu+1},$ that is $i\geq k+1-s_{\mu+1}$.
We have to consider several cases.
 	\begin{enumerate}[label=(\roman*)]
 	\item  $ k+1-s_{\mu+1}\leq 1$, that is   $s_{\mu+1}\geq  k.$  Then  following the case (A) we have 
\begin{itemize}
\item   If either $ r'_\mu >h$ or  $ r'_\mu =h, s_{\mu }-s_{\mu+1}=1$  then we have  $$\widetilde{PF(S)}_2=\{x_i x_k ^{\sigma_{\mu }-1} x_{k+1}^{p_{\mu+1 }-p_{\mu} -1}, i=1,\dots ,k\}.$$
 \item  If $ r'_\mu =h, s_{\mu }-s_{\mu+1}>1$ then we have  $$\widetilde{PF(S)}_2=\{x_1 x_k ^{\sigma_{\mu }-1} x_{k+1}^{p_{\mu+1 }-p_{\mu} -1}\}.$$
 \end{itemize}
 		\item $1<s_{\mu+1}<  k $. We have $1<s_{\mu }-s_{\mu+1}= (\sigma_{\mu }-1)k+(k+1-s_{\mu+1})$ so $k+1-s_{\mu+1}=\widetilde{\rho }\geq 2$. That is
 		$x_i x_k ^{\sigma_{\mu }-1} x_{k+1}^{p_{\mu+1 }-p_{\mu}}\in \mathcal V$ for $i\geq \widetilde{\rho }\geq 2$.
\begin{itemize}
\item   If $ r'_\mu >h$  then we have $$\widetilde{PF(S)}_2  =\{x_i x_k ^{\sigma_{\mu }-1} x_{k+1}^{p_{\mu+1 }-p_{\mu} -1}, i=\widetilde{\rho },\dots ,k  \}.$$
 \item  If $ r'_\mu =h$ then we have  $\widetilde{PF(S)}_2=\emptyset.$
 \end{itemize}
 		\item $s_{\mu+1}=1 $. We have  $i+(\sigma_{\mu }-1)k\geq s_{\mu }-s_{\mu+1}=\sigma_{\mu }k$ if and only if $i=k$.
\begin{itemize}
\item   If $ r'_\mu >h$  then $$\widetilde{PF(S)}_2  =\{ x_k ^{\sigma_{\mu }} x_{k+1}^{p_{\mu+1 }-p_{\mu} -1} \}.$$
 \item  If $ r'_\mu =h$ then we have  $\widetilde{PF(S)}_2=\emptyset.$
 \end{itemize} 
 	\end{enumerate}
\item Suppose $\rho_{\mu }>1$. We have $x_i x_k ^{\sigma_{\mu }} x_{k+1}^{p_{\mu+1 }-p_{\mu}}\in \mathcal V$ if and only if  $i+\sigma_{\mu }k\geq s_{\mu }-s_{\mu+1 } $ that is $i \geq \rho _{\mu }-s_{\mu+1 } $. We have to consider two cases. 
 	\begin{enumerate}[label=(\roman*)]
 		\item If $\rho _{\mu }-s_{\mu+1 }\leq 1 $ then we have $i \geq \rho _{\mu }-s_{\mu+1 } $ for any $i\geq 1$.
 		Hence $$\widetilde{PF(S)}_2  =\{x_i x_k ^{\sigma_{\mu }} x_{k+1}^{p_{\mu+1 }-p_{\mu} -1}, i=1,\dots ,\rho_{\mu }-1\}.$$
 		
 		\item If $\rho _{\mu }-s_{\mu+1 } >1 $ then we have $\widetilde{\rho }=\rho _{\mu }-s_{\mu+1 }$ and 
 		$$\widetilde{PF(S)}_2  =\{x_i x_k ^{\sigma_{\mu }} x_{k+1}^{p_{\mu+1 }-p_{\mu} -1}, i=\widetilde{\rho },\dots ,\rho_{\mu }-1\}.$$
 	\end{enumerate}
\end{enumerate}	
\end{enumerate} 
 \end{proof} 

\section{Symmetric AAG-semigroups}
In this section we will describe all symmetric AAG-semigroups. 
Recall that we always assume  $(a,d)=1$. Symmetric semigroups generated by 4 elements are well known, see for example \cite{Br}, so we can assume $k\geq 3$. When $h=1$, the case $S$  symmetric was studied  in \cite{RR}.

 \begin{theorem}\label{tGor}  Let $S$ be an AAG semigroup. Suppose that $k\geq 3$ and either  $r'_\mu\geq h$ or $ \rho_\mu =0$. 
 We have  $S$ is symmetric if and only   if either 
 \begin{enumerate}  
 \item     $a=(\sigma k+2) p',\  d=p'r-p\hat r, \ c= -(\sigma k+2) \hat r ,$\\ where $\sigma\geq 1,1\leq p<p',  h\geq 1, \hat r<-1$ 
with $\gcd(p',\hat r)=1, r+h\sigma >0$. 
The Frobenius number is $(ha+d)+ \sigma (ha+kd)+ c(p'-1)-a$, or 
\item   $a=(\sigma k+2) p'- \sigma' k p, \ d=p'r+ph\sigma', \ c= \sigma' k r +(\sigma k+2)\sigma'h,$\\ where $\sigma,\sigma ', p',p , r,h$ are integers such that $\sigma\geq \sigma ' \geq 2, p'>p >0, r+h(\sigma+1) >0$. 
The Frobenius number is $(ha+d)+ \sigma (ha+kd)+ c(p'-p-1)-a $, or 
\item   $a=(\sigma k+2) p'- (\sigma k+1) p, \ d=p'r+ph(\sigma+1), \ c= (\sigma k +1) r +(\sigma k+2)(\sigma+1)h,$\\  where $\sigma,\sigma ', p',p , r,h$ are integers such that $\sigma\geq 1,p'>p >0, r+h(\sigma+1) >0$. 
   The Frobenius number is $(ha+d)+ \sigma (ha+kd)+ c(p'-p-1)-a $, or
\item  $a=(\sigma k+1)p'-(k-1)p , \ d=-p'h\sigma-p\hat r, \ c=-(k-1)h\sigma-(\sigma k+1)\hat r,$\\  where $\sigma\geq 1, p'>p>0,   \hat r<-h$.
The Frobenius number is $(ha+d)+ (\sigma-1) (ha+kd)+ c(p'-1)-a $.
\end{enumerate}  
\end{theorem} 

\begin{proof} In all the cases we get a  \gbb by applying Theorem \ref{AAG-1}. The proof consists of examining all possibles cases of Theorem \ref{tPSF}. We will omit referring each time.
 \begin{enumerate}\item $s_{\mu+1 }=0$. Then $r'_{\mu+1 }\not=0$ and we have that $\widetilde{PF(S)}=\widetilde{PF(S)}_1$ with $\widetilde{\sigma }=
\sigma _{\mu }, \widetilde{\rho }=\rho _{\mu }$.  $\card (\widetilde{PF(S)})=1$ if and only if we are in case 2d) with $\rho _{\mu }=2$, and $\widetilde{PF(S)}=\{x_1 x_k^{\sigma _{\mu }} x_{k+1}^{p_{\mu +1}-1}\}$. We have $s_{\mu }=\sigma k+2$ for some $\sigma \geq 1$. We set $p\defeq p_{\mu }, p'\defeq p_{\mu+1 }, r\defeq r_{\mu }, \hat r\defeq r_{\mu+1 }$. We have the table 
$$\begin{tabular}{lll|l}
   s & p & r & r'\\
    \hline 
   $\cdots$& $\cdots$& $\cdots$&$\cdots$\\
   $\sigma k+2$ & $p$ & $r$& \\
   $0$ & $p'$ & $\hat r$ & \\
\end{tabular}$$
with the conditions $0<p<p', r+h(\sigma +1)>0, \hat r<-1$. By Lemma 2.2.4 of \cite{Mo2}  we get 
$$a=(\sigma k+2) p', d=p'r-p\hat r, c= -(\sigma k+2) \hat r .$$ Since $\gcd(a,d)=1$ we have $\gcd(p',\hat r )=1$. The Frobenius number is 
$\varphi (x_1 x_k^\sigma x_{k+1}^{p'-1})-a$.

 \item  $s_{\mu+1 }\not=0$ and  $\card (\widetilde{PF(S)}_1)= 0, \card (\widetilde{PF(S)}_2)= 1$. We have to consider the cases 1a) or 1c) and one of the cases 4ii),5iii),6ii) 7i),7ii). 
\begin{enumerate}\item 1a) or 1c) and 4ii). $\card (\widetilde{PF(S)}_2)= 1$ implies that $\widetilde{\rho }=k-1$. By 4ii) we have  
 $\rho_{\mu}=0,$  so   $\rho_{\mu+1 }=1$, case 1a) is discarded, and case 1c) is possible if and only if  $k=2$.

\item 1a) and 5iii) we have $\rho_{\mu }=1,s_{\mu+1 }=1= \rho_{\mu+1 }$ a contradiction with 1a).
\item 1a) and 6ii) we have $r'_{\mu+1 }=0, \rho_{\mu+1 }=0,\rho_{\mu }=1, r'_{\mu}=h,  $ so $s_{\mu }=\sigma _{\mu }k+1, s_{\mu+1 }=\sigma_{\mu +1}k $. We set $\sigma \defeq \sigma _{\mu }, \sigma '\defeq \sigma _{\mu+1 },p\defeq p_{\mu }, p'\defeq p_{\mu+1 } $.  But $r'_{\mu+1 }=0= r_{\mu+1 }+h\sigma' $ so  $r_{\mu+1 }=-h\sigma' $, also $r'_{\mu}=h= r_{\mu}+h(\sigma +1) $ so  $r_{\mu}=-h\sigma $. We should have the table 
$$\begin{tabular}{lll|l}
   s & p & r & r'\\
    \hline 
  $\cdots$& $\cdots$& $\cdots$&$\cdots$\\
   $\sigma k+1$ & $p$ & $-h \sigma $& \\
   $\sigma' k$ & $p'$ & $-h\sigma'$ & \\
\end{tabular}$$ 
By Lemma 2.2.4 of \cite{Mo2}  we get $$a=(\sigma k+1) p'- \sigma' k p, \ d=-h\sigma p'+ph\sigma', \ c= -h\sigma \sigma' k  +(\sigma k+1)\sigma 'h $$  for some $\sigma > \sigma',  p'>p>0$.  By simple computations we have
$c=h\sigma', a_k=h p', a_{k-1}= h[p'(1+\sigma -\sigma ')+ \sigma '(p'-p)]=(1+\sigma -\sigma ')a_k+ (p'-p)a_{k+1} $, that is the embedding dimension of $S$ is less than $k+2$, a contradiction.
\item 1a) and 7i) we have $\rho_{\mu }=2,\rho_{\mu+1 }=0 $, and $\widetilde{PF(S)}=\{x_1 x_k^{\sigma _{\mu }} x_{k+1}^{p_{\mu +1}-p_\mu -1}\}$. So we have  
$s_{\mu }=\sigma _{\mu }k+2, s_{\mu+1 }=\sigma_{\mu +1}k, \sigma_{\mu +1}\geq 2$. We set $\sigma \defeq \sigma _{\mu }, \sigma '\defeq \sigma _{\mu+1 },p\defeq p_{\mu }, p'\defeq p_{\mu+1 }, r\defeq r_{\mu }, $ since $r'_{\mu+1 }=0= r_{\mu+1 }+h\sigma' $ we have $r_{\mu+1 }=-h\sigma' $ and the table
$$\begin{tabular}{lll|l}
   s & p & r & r'\\
    \hline 
  $\cdots$& $\cdots$& $\cdots$&$\cdots$\\
   $\sigma k+2$ & $p$ & $r$& \\
   $\sigma' k$ & $p'$ & $-h\sigma'$ & \\
\end{tabular}$$
By Lemma 2.2.4 of \cite{Mo2}  we get $$a=(\sigma k+2) p'- \sigma' k p, \ d=p'r+ph\sigma', \ c= \sigma' k r +(\sigma k+2)\sigma 'h $$  for some $\sigma \geq \sigma'\geq 2, p'>p>0,r+h(\sigma +1)>0$.
The Frobenius number is $\varphi (x_1 x_k^\sigma x_{k+1}^{p'-p-1})-a$.
\item 1a) and 7ii) Since $s_{\mu+1 }<\rho_{\mu }-1$ we have $s_{\mu+1 }=\rho_{\mu+1 }=0$ by 1a), a contradiction.

\item 1c) and 5iii), 6i) or 7ii) are easily discarded. So we consider case 1c) and 7i).\\
We have $\card (\widetilde{PF(S)}_2)=\rho_{\mu }-1=1$ if and only if $\rho_{\mu }=2$. Hence $\widetilde{PF(S)}=\{x_1 x_k^{\sigma _{\mu }} x_{k+1}^{p_{\mu +1}-p_\mu -1}\}$By hypothesis $s_\mu -s_{\mu+1 }=1$  we have  
$s_{\mu }=\sigma _{\mu }k+2, s_{\mu+1 }=\sigma _{\mu }k+1 $ for some $\sigma \geq 1$. Set $p\defeq p_{\mu }, p'\defeq p_{\mu+1 }, r\defeq r_{\mu }, $ since $r'_{\mu+1 }=0= r_{\mu+1 }+h(\sigma+1) $ we have $r_{\mu+1 }=-h(\sigma+1) $ and the table: 
$$\begin{tabular}{lll|l}
   s & p & r & r'\\
    \hline 
  $\cdots$& $\cdots$& $\cdots$&$\cdots$\\
   $\sigma k+2$ & $p$ & $r$& \\
   $\sigma k+1$ & $p'$ & $-h(\sigma+1)$ & \\
\end{tabular}$$ 
 By Lemma 2.2.4 of \cite{Mo2}  we get\\
 $ a=(\sigma k+2) p'-(\sigma k+1) p, \ d=p'r+ph(\sigma +1), \ c= (\sigma k+1) r +(\sigma k+2)(\sigma +1)h $
for some $\sigma\geq 1,p'>p>0,
r>-h(\sigma+1)$.
The Frobenius number is $\varphi (x_1 x_k^\sigma x_{k+1}^{p'-p-1})-a$.
\end{enumerate}

\item 
 $s_{\mu+1 }\not=0$ and   $\card (\widetilde{PF(S)}_1)= 1, \card (\widetilde{PF(S)}_2)= 0$. We are in case 6iii), we have $\rho_{\mu }=1,r'_{\mu }=h,  0<s_{\mu+1 }<k$ 
so $r'_{\mu+1 }<0$ and  $s_\mu -s_{\mu+1 }=  (\sigma _{\mu}-1)k+(k-s_{\mu+1 }+1) $. If $s_{\mu+1 }=1$ we have $\widetilde{\rho }=0$, we are in case 2a) and  $\card (\widetilde{PF(S)}_1)= 1$ implies $k=2$. If $s_{\mu+1 }>1$ we have $1<\widetilde{\rho }=k-s_{\mu+1 }+1<k $, so we are in case 2d).   $\card (\widetilde{PF(S)}_1)= 1$ implies $\widetilde{\rho }=2$, that is $s_{\mu+1 }=k-1 $, and $\widetilde{PF(S)}=\{x_1 x_k^{\sigma _{\mu }-1} x_{k+1}^{p_{\mu +1}-1}\}$. On the other hand $r'_{\mu }=h=r_{\mu }+h( \sigma_{\mu} +1)$  so 
$r_{\mu }=-h \sigma_{\mu} $. Set $\sigma \defeq \sigma_{\mu}, p\defeq p_{\mu},p'\defeq p_{\mu+1}, \hat r\defeq r_{\mu+1}$ we have the table
 $$\begin{tabular}{lll|l}
   s & p & r & r'\\
    \hline 
   $\cdots$& $\cdots$& $\cdots$&$\cdots$\\
   $\sigma k+1$ & $p$ & $-h \sigma $& \\
   $k-1$ & $p'$ & $\hat r$ & \\
\end{tabular}$$
 with 
$\sigma\geq 1, p'>p>0, \hat r<-h$. By Lemma 2.2.4 of \cite{Mo2}  we get   $$a=(\sigma k+1)p'-(k-1)p , d=-p'h\sigma-p\hat r, c=-(k-1)h\sigma-(\sigma k+1)\hat r.$$ Note that if  $d>0$ then $\hat r<-(p'/p)h\sigma$.
The Frobenius number is $\varphi (x_1 x_k^{\sigma -1}x_{k+1}^{p'-1})-a$. 
\end{enumerate}
\end{proof} 

\section{Almost Symmetric AAG-semigroups, $d\in \fz^*$}

Almost symmetric semigroups generated by 4 elements are well known, see for example \cite{Mos}, \cite{E}, so we can assume $k\geq 3$. 
 The case $S$ almost symmetric of type 2 (pseudo symmetric) was studied in \cite{GRR}. The proofs consist of examining all possibles cases of Theorem \ref{tPSF}. We will omit referring each time. Note that by evaluation of $\varphi $ the elements in $\widetilde{PF(S)}_i$ are ordered by increasing order if $d>0$ and by decreasing order if $d<0$.
 
The following lemma will be very useful to study almost symmetric of AAG semigroups with ratio a relative integer.

\begin{lemma}\label{symm-M} $(1)$ If 
$ x_0^\alpha  N_1 -x_0^\beta  N_2 \in I$ for some monomials $N_1$,  $N_2$ where $N_2\in \widetilde{\Ap(S,a)}$ then $\alpha \leq \beta $. \\
$(2)$ Suppose $t(S)>1$. Let $N_1,N_2\in \widetilde{PF(S)}$ such that  
 $N_1 N_2- x_0 M\in I,$ where $M\in \widetilde{\Ap(S,a)}$. If  
$ N_1 N_2 -x_0^\lambda N \in I$ for some monomial $N$ and $\lambda \in \fn^*$ then $\lambda=1$.  In particular,  if $S$ is an AAG semigroup and 
 $N_1=x_i N'_1, N_2=x_j N'_2$ with $i+j\leq k$ then $h=1$.
\end{lemma}

\begin{proof} The first claim is the definition of $\widetilde{\Ap(S,a)}$. The second claim follows from the first. The third claim follows from the first since $x_ix_j-x_0^h x_{i+j}\in I$ when $i+j\leq k$.
\end{proof}
As the description of almost symmetric AAG semigroups is quite long, we will divide in three subsections.

\begin{subsection}{Almost symmetric AAG-semigroups with $\widetilde{\rm Frob(S) }\in\widetilde{PF(S)_\varepsilon }$ with $\varepsilon \in \{1,2\},\card(\widetilde{PF(S)_\varepsilon })\geq 2 $}

\begin{theorem}\label{almostsym00}
With the above notations, suppose $d>0$ and  either  $r'_\mu\geq h$ or $ \rho_\mu =0$.
 Suppose that $S$ is almost symmetric of type $\geq 2$. If $\widetilde{{\rm Frob(S)} }\in\widetilde{PF(S)_\varepsilon }$ with $\varepsilon \in \{1,2\},\card(\widetilde{PF(S)_\varepsilon })\geq 2 $ then we have
 $$ h=1,k\geq 3 \ \ odd, a=k+2, d\in 2\fn^*, c=a+(d/2)k,   \gcd(a,d)=1.$$ 
 Moreover 
$\widetilde{PF(S) }=\{ x_1,\dots ,x_k\}\cup \{x_{k+1}\} $, $t(S)=k+1$, $\widetilde{{\rm Frob(S)} }=x_k,F(S)=kd$.
 \end{theorem}

 \begin{proof}
 There exists $2\leq l\leq k, $ $M$ a monomial such that
 $\widetilde{{\rm Frob(S)} }=x_lM$ and $x_ {l-1}M\in   \widetilde{PF(S)_{\varepsilon } }$.  Since $S$ is almost symmetric, there exists $M_1 \in   \widetilde{PF(S)  }$ such that $M_1 x_{l-1}M -x_0x_lM\in I$, which implies  $ x_{l-1}M_1  -x_0x_l\in I$. We multiply by $x_1$ and using the \gbb we get $ x_0^h x_{l}M_1  -x_0x_lx_1\in I$ that implies $h=1$, $M_1-x_1\in I$, if $M_1\not= x_1$ then the embedding dimension of $S$ is less than $k+2$ contrary to our hypothesis, therefore $M_1=x_1\in \widetilde{PF(S)  }$. We have to examine all the possibles cases in Theorem \ref{tPSF} such that $ x_1\in \widetilde{PF(S)}$. Since $\sigma _{\mu}\geq 1, p_{\mu +1}\geq 2$ the possible cases are  4), 5) or 6). In particular we have $\sigma _{\mu}=1.$ 
The case 4)  implies $s_\mu =k,$ a contradiction since we have $s_\mu >k$. The cases 5) and 6) imply    $\widetilde{PF(S)_2}=\{ x_1,\dots ,x_k\}$ with $s_\mu =k+1, s_{\mu +1}=k,\widetilde{\sigma }=0,  \widetilde{\rho }=1, r'_{\mu +1}<0$. The  only possible case is  2b), so we have $\widetilde{PF(S)_1}= \{x_{k+1}^{p_{\mu+1}-1}\}$, $\varepsilon =2, \widetilde{{\rm Frob(S)} }=x_k$. Since $h=1$, by the \gbb we have $ x_i x_{k-i}-x_0 x_k\in I$, $S$ is almost symmetric if and only if  
$x_{k+1}^{2(p_{\mu+1}-1)}-x_0 x_k\in I$ so that 
$x_{0}^{-r'_{\mu+1}}x_k x_{k+1}^{p_{\mu+1}-2}-x_0 x_k\in I,$ which implies   $p_{\mu +1}=2,  r'_{\mu +1}=-1$ and $p_{\mu}=1$, so $r_{\mu+1}=r'_{\mu+1}-1=-2$, we have the table 
$$\begin{tabular}{lll|l}
   s & p & r & r'\\
    \hline 
  $\cdots$& $\cdots$& $\cdots$&$\cdots$\\
   $k+1$ & 1 & $r_{\mu}$& $r'_{\mu}$\\
   $k$ & $2$ & $-2$ & $-1$ \\
\end{tabular}$$ 
Note that $r'_{\mu}=h=1$  if and only if $r_{\mu}=-1$, which implies $d=0$, so the case 6) is not possible. The case 5) is possible and we have $a=k+2, d=2r_{\mu}+2, c=k(r_{\mu}+2)+2$ with $k$ odd, $k\geq 3, r_{\mu}\geq 0$, $\gcd(a,d)=1$, that is 
$$ h=1,k\geq 3  {\rm \ odd},  a=k+2, d\in 2\fn^*, c=a+(d/2)k.$$
 Moreover we have 
$\widetilde{PF(S) }=\{ x_1,\dots ,x_k\}\cup \{x_{k+1}\} $, $t(S)=k+1$, $F(S)=kd$.
\end{proof}

\begin{theorem}\label{almostsym0}
With the above notations, suppose that $d<0$ and either  $r'_\mu\geq h$ or $ \rho_\mu =0$.
 If  $S$ is almost symmetric, and $\widetilde{{\rm Frob(S)} }\in\widetilde{PF(S)_\varepsilon }$ with $\varepsilon \in \{1,2\},\card(\widetilde{PF(S)_\varepsilon })\geq 2 $ then we have:
$$
a=(\sigma k+1)p-\sigma k, \  
d=-h\sigma (p-1)+1<0, \ 
c=h\sigma +\sigma k+1,  
$$
with $gcd(a,d)=1$, $h\geq 2,k\geq 3, \sigma \geq 1, p\geq 2$. Moreover  $t(S)=k+1$ and $ \widetilde{PF(S)_1} =\{  x_{k+1}^{p-1}\}$, $ \widetilde{PF(S)_2} =\{x_1  x_k^{\sigma -1 } x_{k+1}^{p-2},\dots ,x_k  x_k^{\sigma-1} x_{k+1}^{p-2} \}$, $\widetilde{{\rm Frob(S)}}=x_1  x_k^{\sigma -1 } x_{k+1}^{p-2}$.
 \end{theorem}
 
 \begin{proof}

{\bf I)} Suppose $\card(\widetilde{PF(S)_2})\geq 2$ and $\widetilde{{\rm Frob(S)}}\in \widetilde{PF(S)_2}$. By Theorem \ref{tPSF} we have that either $\widetilde{{\rm Frob(S)}}= x_1M$ or $\widetilde{{\rm Frob(S)}}= x_{\widetilde{\rho }}M$ for some monomial $M$ and  $\widetilde{\rho }\geq 2$.\\ 
\noindent{\bf I-A) } Suppose that $\widetilde{{\rm Frob(S)}}= x_1M$, we have that $x_2M\in \widetilde{PF(S)_2}$, so there exists some monomial $M_1\colon =L_i x_k^\alpha 
x_{k+1}^\beta \in  \widetilde{PF(S)}$ such that $$ x_2M M_1 -x_0x_1M\in I, x_2 M_1 -x_0x_1\in I.$$ If $0<i<k-1$ then we have  a contradiction by Lemma \ref{symm-M}. If $i=k-1$ 
then $x_{2}x_{k-1}-x_1x_k\in I$, it leads to 
$ x_k x_k^\alpha  x_{k+1}^\beta-x_0\in I, $ contradicting   that the embedding dimension of $S$ is $k+2$. So $M_1\colon = x_k^\alpha 
x_{k+1}^\beta.$ Suppose $\alpha >0$. We examine two cases:
\begin{enumerate}
\item  $M_1\in \widetilde{PF(S)_1}$. We are in case 2c) of Theorem \ref{tPSF}. So ${\widetilde{\rho }}=1$,  $M_1=x_k^{\widetilde{\sigma }}x_{k+1}^{p _{\mu +1}-1}$ with $\widetilde{\sigma }>0, r'_{\mu +1}<0$. We have $x_1x_k^{\widetilde{\sigma }}x_{k+1}^{p _{\mu +1}-p _{\mu}}-
x_0^{\widetilde{r}}\in I,$ so $x_2x_k^{\widetilde{\sigma }}x_{k+1}^{p _{\mu +1}-p _{\mu}}-
x_0^{\widetilde{r}-h}x_1\in I$, we also have $x_2 M_1-x_0x_1\in I$.
Hence 
$$ x_0^{\widetilde{r}-h} x_1 x_{k+1}^{p _{\mu} -1}-x_0 x_1\in I,$$
which implies $p _{\mu}=1, \widetilde{r}=h+1$. By applying  Lemma \ref{widetildehorizontal} we get 
\begin{enumerate}
\item  
If $\rho_\mu =0 $  then $\widetilde{r}=r'_\mu -r'_{\mu +1}+h=h+1$ which is not possible since $r'_\mu\geq 1, r'_{\mu +1}<0$.
\item  
If $\rho_\mu =1$ then  $\widetilde{r}=r'_\mu-r'_{\mu +1}=h+1$. It is possible only if $r'_\mu=h$, $r'_{\mu +1}=-1$, which implies that we are in case 6i). Hence $s_{\mu }-s_{\mu +1}=1,$  contradiction to 
 $\widetilde{\sigma }>0$. 
\item  
If $\rho_\mu >1$ then we have $\rho_\mu> \rho _{\mu +1}$ and  $\widetilde{r}=r'_\mu-r'_{\mu +1}+h=h+1$, which is not possible since $r'_\mu\geq h, r'_{\mu +1}<0$.
\end{enumerate}

\item $M_1\in \widetilde{PF(S)_2}$. We are  in cases  5i), 5ii) or 6i). We have $M_1=x_k^{\sigma _\mu }x_{k+1}^{p _{\mu +1}-p _\mu -1} $  so  
$x_2x_k^{\sigma _\mu }x_{k+1}^{p _{\mu +1}-p _\mu -1} -x_0x_1\in I,$ but by the \gbb of $I$ we have 
$x_2x_k^{\sigma _\mu }- x_0^{r'_\mu -h} x_1 x_{k+1}^{p _\mu }\in I$, hence $x_0^{r'_\mu -h} x_1 x_{k+1}^{p _{\mu +1}-1 }-x_0x_1\in I,$ that is 
$x_0^{r'_\mu -h}  x_{k+1}^{p _{\mu +1}-1 }-x_0\in I$, where $p _{\mu +1}>1$, which  implies that the embedding dimension of $S$ is less than $k+2$, a contradiction. 
\end{enumerate} 

To resume we have proved that $M_1=x_{k+1}^\beta $. The only possible case is 2b) of  Theorem \ref{tPSF}, we have  $M_1=x_{k+1}^{p _{\mu +1}-1}, s_{\mu }-s_{\mu +1}=1$ and 
\begin{equation}\label{neg-case1}
	 x_2 x_{k+1}^{p _{\mu +1}-1}-x_0x_1\in I, x_1 x_{k+1}^{p _{\mu +1}-p_\mu }-x_0^{\widetilde{r}}\in I.
  \end{equation} 
So that 
$$ x_0^{\widetilde{r}-h} x_1 x_{k+1}^{p _{\mu} -1}-x_0x_1 \in I,$$
which implies  $p _{\mu}=1, \widetilde{r}=h+1$. Now we examine all cases:

 If $\rho_{\mu }=0$ or $\rho_{\mu }>1$  then $\widetilde{r }=r'_{\mu }-r'_{\mu +1}+h=h+1$, which is impossible. 

If $\rho_{\mu }=1$ then $\widetilde{r }=r'_{\mu }-r'_{\mu +1}=h+1$, which implies $r'_{\mu}=h, r'_{\mu +1}=-1$. We are in the case 6i)
$$ \widetilde{PF(S)}_2=\{x_i x_k ^{\sigma_{\mu }-1} x_{k+1}^{p_{\mu+1 }-p_{\mu} -1}, i=1,\dots ,k\}.$$ Note that $x_3 x_k^{\sigma _{\mu}}-x_2 x_{k+1}\in I$  implies 
$$ (x_{3+i}  x_k^{\sigma _\mu-1 }x_{k+1}^{p _{\mu+1 }-2})(x_{k-i}  x_k^{\sigma _\mu-1 }x_{k+1}^{p _{\mu+1 }-2})
-x_0 x_1 x_k^{\sigma _\mu-1 }x_{k+1}^{p _{\mu+1 }-2}\in I$$ for $i=0,\ldots, k-3$, which verifies that $S$ is almost symmetric. We set $p=p _{\mu+1 }, \sigma=\sigma _{\mu}$. The only possible case is given by the table 
$$\begin{tabular}{lll|l}
   s & p & r & r'\\
    \hline 
   $a$& $0$& $d$& \\
   $\sigma k+1$ & $1$ & $-h\sigma $&$h$\\
   $\sigma k$ & $p$ & $-h\sigma-1$ & $-1$\\
\end{tabular}$$
  By Lemma 2.2.4 of \cite{Mo2}  we get
  $$   a=(\sigma k+1)p-\sigma k;   
d=(p-1)(-h\sigma)+1;   
  c=h\sigma +\sigma k+1 $$ 
with  $\gcd(a,d)=1, h\geq 2,p\geq 2, \sigma \geq 1$.\\
\noindent{\bf I-B) } Suppose that $\widetilde{{\rm Frob(S)}}= x_{\widetilde{\rho }}M$ with $\widetilde{\rho }>1$, we have that $x_{\widetilde{\rho }+1}M\in \widetilde{PF(S)_2}$, so there exists some monomial $M_1 \in  \widetilde{PF(S)}$ such that $$ x_{\widetilde{\rho }+1}M M_1-x_0x_{\widetilde{\rho }}M\in I.$$ 
The possible cases  in Theorem \ref{tPSF} are 4ii), 5ii) or 7ii).
\begin{enumerate}
\item   
If $\rho _\mu =0$, we are in case 4ii). First suppose that $M_1\in \widetilde{PF(S)_2}$. It implies $M_1=x_ix_{k}^{\sigma _{\mu }-1}
x_{k+1}^{p _{\mu+1 }-p_\mu -1}$ for some $i=\widetilde{\rho }+1,\ldots,k-1$
 so that $x_{\widetilde{\rho }+1}x_{i} x_{k}^{\sigma _{\mu }-1}x_{k+1}^{p _{\mu+1 }-p_\mu -1}-x_0x_{\widetilde{\rho }}\in I$. If
 $\widetilde{\rho }+1+i\leq k $ then we have a contradiction.  If
 $\widetilde{\rho }+1+i> k $ then  we can use   $x_{k}^{\sigma _{\mu }}-x_{0}^{r'_{\mu }}x_{k+1}^{p_{\mu }}\in I$. So  we have 
 $x_{0}^{r'_{\mu }}x_{\widetilde{\rho }+1+i-k} x_{k+1}^{p _{\mu+1 }-1}-x_0x_{\widetilde{\rho }}\in I$, but $r'_{\mu }\geq 1$, contrary to our hypothesis that the embedding dimension of $S$ is $k+2$.\\
 Now suppose that $M_1\in \widetilde{PF(S)_1}$. Since we are in case 4ii) we have $s _{\mu +1}<k-1$ which implies $r'_{\mu +1}<0$ and $\widetilde{\sigma}=\sigma _{\mu }-1, \widetilde{\rho }=k-s_{\mu +1}\geq 2.$   
Hence the only possible case for $M_1$ is 2d),   $M_1=x_{i} x_{k}^{\sigma _{\mu }-1}x_{k+1}^{p _{\mu+1 }-1}$ for some $1\leq i\leq \widetilde{\rho }-1$, but $\widetilde{{\rm Frob(S)}}= x_{\widetilde{\rho }} x_{k}^{\sigma _{\mu }-1}x_{k+1}^{p _{\mu+1 }-p_\mu -1}$, hence  $\varphi ( M_1)>\varphi ( \widetilde{{\rm Frob(S)}})$ a contradiction.

\item  If $\rho _\mu =1$, we are in the case 5ii).  
First suppose that $M_1\in \widetilde{PF(S)_2}$. We have  $M_1=x_{k}^{\sigma _{\mu }}
x_{k+1}^{p _{\mu+1 }-p_\mu -1}$  and  $ x_{\widetilde{\rho }+1} M_1- x_0 x_{\widetilde{\rho }}\in I,$   but 
$x_{\widetilde{\rho }+1} x_{k}^{\sigma _{\mu }}-  x_0^{r'_{\mu }-h}x_{\widetilde{\rho }}x_{k+1}^{p_\mu}\in I$ which implies 
$ x_0^{r'_{\mu }-h}x_{\widetilde{\rho }}x_{k+1}^{p _{\mu+1 } -1}-x_0x_{\widetilde{\rho }}\in I$. Let recall that   $ p_{\mu +1}\geq 2$, which leads to a contradiction. 

Now suppose that $M_1\in \widetilde{PF(S)_1}$. Since we are in the case 5ii), we have $1<s _{\mu +1}<k$ which implies $r'_{\mu +1}<0$ and $\widetilde{\sigma}=\sigma _{\mu }-1, \widetilde{\rho }=k+1-s_{\mu +1}>1$.  Hence the only possible case for $M_1$ is 2d),  $M_1=x_{i} x_{k}^{\sigma _{\mu }-1}x_{k+1}^{p _{\mu+1 }-1}$ for some $1\leq i\leq \widetilde{\rho }-1$, but $\widetilde{{\rm Frob(S)}}= x_{\widetilde{\rho }} x_{k}^{\sigma _{\mu }-1}x_{k+1}^{p _{\mu+1 }-p_\mu -1}$. Hence  $\varphi ( M_1)>\varphi ( \widetilde{{\rm Frob(S)}})$ a contradiction. 

\item  If $\rho _\mu >1$, we are in case 7ii).  
First suppose that $M_1\in \widetilde{PF(S)_2}$. It implies $M_1=x_{\rho _\mu -1}x_{k}^{\sigma _{\mu }}
x_{k+1}^{p _{\mu+1 }-p_\mu -1}
$ so that $x_{\widetilde{\rho }+1}x_{\rho _\mu -1} x_{k}^{\sigma _{\mu }}
x_{k+1}^{p _{\mu+1 }-p_\mu -1}-x_0x_{\widetilde{\rho }}\in I$. If  $\widetilde{\rho }+ \rho _\mu\leq k $ we have $x_{\widetilde{\rho }+1}x_{\rho _\mu -1} -x_0^h x_{\widetilde{\rho }+\rho _\mu}\in I,$ which leads to a contradiction since $h\geq 2$. If  $\widetilde{\rho }+ \rho _\mu> k$ then we have $x_{\widetilde{\rho }+1}x_{\rho _\mu -1} - x_{\widetilde{\rho }+ \rho _\mu-k}x_k\in I$ and 
$x_{k}^{\sigma _{\mu }+1}- x_0^{r'_{\mu }-h}x_{k-\rho _\mu} x_{k+1}^{p_\mu}\in I$. So 
$x_0^{r'_{\mu }-h}x_{\widetilde{\rho }+ \rho _\mu-k} x_{k-\rho _\mu} x_{k+1}^{p _{\mu+1 }-1}-  x_0x_{\widetilde{\rho }}\in I$. But  
$x_{\widetilde{\rho }+ \rho _\mu-k} x_{k-\rho _\mu} - x_0^hx_{\widetilde{\rho }} \in I$
 implies that 
$ x_0^{r'_{\mu }}x_{\widetilde{\rho }}x_{k+1}^{p _{\mu+1 } -1}-x_0x_{\widetilde{\rho }}\in I$ with $r'_{\mu }\geq h\geq 2$ which leads to a contradiction.

Now suppose that $M_1\in \widetilde{PF(S)_1}$, since we are in then case 7ii) we have $1<s _{\mu +1}< \rho _\mu-1$ which implies $r'_{\mu +1}<0$, and 
$\widetilde{\sigma}=\sigma _{\mu }, \widetilde{\rho }=\rho _\mu-s _{\mu +1}>1$. Hence the only possible case for $M_1$ is 2d),  $M_1=x_{i} x_{k}^{\sigma _{\mu }}x_{k+1}^{p _{\mu+1 }-1}$ for some $1\leq i\leq \widetilde{\rho }-1$, but  $\widetilde{{\rm Frob(S)}}= x_{\widetilde{\rho }} x_{k}^{\sigma _{\mu }}x_{k+1}^{p _{\mu+1 }-p_\mu -1}$. Hence  $\varphi ( M_1)>\varphi ( \widetilde{{\rm Frob(S)}})$ a contradiction.
\end{enumerate}

{\bf II)} Suppose $\card(\widetilde{PF(S)_1})\geq 2,  \widetilde{{\rm Frob(S)}}\in \widetilde{PF(S)_1}$.

By Theorem \ref{tPSF} we have    $\widetilde{{\rm Frob(S)}}= x_1M$ for some monomial $M=x_k^u x_{k+1}^{p _{\mu+1 }-1}$, where $u=\widetilde{\sigma} $ in case 1e) or 2d) and $u=\widetilde{\sigma}-1 $ in  cases 1b), 1d), 2a), or 2b). We have that $x_2M\in \widetilde{PF(S)_1}$, so there exists some monomial $M_1\colon =L_i x_k^\alpha 
x_{k+1}^\beta \in  \widetilde{PF(S)}$ such that $$ x_2M L_i x_k^\alpha 
x_{k+1}^\beta-x_0x_1M\in I.$$  By the same argument as the beginning of part I, we have $M_1\colon = x_k^\alpha 
x_{k+1}^\beta.$ We examine two cases:
 
\begin{enumerate}
\item 
 If $s_{\mu +1}=0$ then we have $\widetilde{PF(S)_2}=\emptyset $ and $r'_{\mu +1}<0$, so $M_1\in \widetilde{PF(S)_1}$. We have to check the cases 2a) 2c) and 2d).
\begin{enumerate}
\item  In the   case 2a) we have $\rho _{\mu }=0, \widetilde{\sigma}=\sigma _{\mu }$, $M_1=x_{k-1}x_k^{\sigma _{\mu }-1} x_{k+1}^{p _{\mu+1 }-1}$, and $S$ is almost symmetric which implies  
$x_2 x_{k-1}M-x_0x_1\in I.$ But  $x_2 x_{k-1}-x_1 x_k\in I $,    hence the condition $ x_k^{\sigma _{\mu }} x_{k+1}^{p _{\mu+1 }-1}-x_0\in I$  implies that the embedding dimension of $S$ is less than $k+2$.

 \item In the   case 2c) we  have $\rho _{\mu }=1,  \widetilde{\sigma}=\sigma _{\mu }$, and $M_1=x_k^{\sigma _{\mu }} x_{k+1}^{p _{\mu+1 }-1}$, $S$ almost symmetric implies  
$x_2 x_k^{\sigma _{\mu }} x_{k+1}^{p _{\mu+1 }-1}-x_0x_1\in I$. But 
 $x_2 x_k^{\sigma _{\mu }}- x_0^{r'_{\mu }-h} x_1   x_{k+1}^{p _{\mu }}\in I$, so  we have 
 $$x_1 x_0^{r'_{\mu }-h}   x_{k+1}^{p _{\mu+1 }+p _{\mu }-1}-x_0x_1\in I, $$ 
 with $p _{\mu+1 }+p _{\mu }-1\geq 2$ which implies that the embedding dimension of $S$ is less than $k+2$.
 
 \item  In the   case 2d) we have $\rho _{\mu }= \widetilde{\rho }>1,  \widetilde{\sigma}=\sigma _{\mu }$, 
$M_1=x_{\rho _{\mu }-1} x_k^{\sigma _{\mu }} x_{k+1}^{p _{\mu+1 }-1}$. Since $1<\rho _{\mu }<k$ we have $x_2 x_{\rho _{\mu }-1} -x_0^h   x_{\rho _{\mu }+1} \in I$. Hence  $x_0^h   x_{\rho _{\mu }+1}x_k^{\sigma _{\mu }} x_{k+1}^{p _{\mu+1 }-1}-x_0x_1\in I $ which leads to a contradiction since $h\geq 2$.
 \end{enumerate}
 
 \item If $s_{\mu +1}\not=0$ then we have two cases, either $M_1\in \widetilde{PF(S)_1} $  or $M_1\in \widetilde{PF(S)_2}  $.\\ 
 $\bullet$ If $M_1=x_k^\alpha  x_{k+1}^{p _{\mu+1 }-1}\in \widetilde{PF(S)_1} $ then the only possible case is 2c). So $r'_{\mu +1}<0,\widetilde{\rho }=1, 
\alpha =\widetilde{\sigma  }>0$. Since $S$ is almost symmetric we have $ x_2 x_k^{\widetilde{\sigma  }} x_{k+1}^{p _{\mu+1 }-1}-x_0 x_1\in I$. By using the \gbb of $I$ we have 
$x_2 x_k^{\widetilde{\sigma  }} x_{k+1}^{p _{\mu+1 }- p_\mu }-x_0^{\widetilde{r}-h}x_1 \in I.$ Hence   
$x_0^{\widetilde{r}-h}x_1 x_{k+1}^{p _{\mu }- 1}- x_0 x_1\in I$ which is possible if and only if $\widetilde{r}-h=1, p _{\mu }= 1$.
 By applying  Lemma \ref{widetildehorizontal} we get
  
\begin{enumerate}
\item  
If $\rho_\mu =0$ then $\widetilde{r}=r'_\mu -r'_{\mu +1}+h=h+1$ which is not possible since $r'_\mu\geq 1, r'_{\mu +1}<0$.

\item  
If $\rho_\mu =1$ then  $\widetilde{r}=r'_\mu-r'_{\mu +1}=h+1$ which is possible only if $r'_\mu=h, r'_{\mu +1}=-1$, which implies that we are in the case 6).

 In case 6i) we have  $s_{\mu }-s_{\mu +1}=1$, a contradiction with 
 $\widetilde{\sigma }>0$.
 
 In case 6ii), since $S$ is almost symmetric there exists $M_2\in \widetilde{PF(S)} $ such that
 $x_1 x_k^{ \sigma_\mu -1} x_{k+1}^{p _{\mu+1 }-p_\mu -1} M_2 -x_0 \widetilde{{\rm Frob(S)}}\in I $.
If  $M_2\in \widetilde{PF(S)}_1 $ then $M_2=x_j M$ for some $j=3,\ldots,k-1$, recall that $\widetilde{{\rm Frob(S)}}=x_1M$, so we have 
$x_j x_k^{ \sigma_\mu -1} x_{k+1}^{p _{\mu+1 }-p_\mu -1}  -x_0 \in I $  for some $j=3,\ldots,k-1$, which is impossible.\\ 
If  $M_2\in \widetilde{PF(S)}_2 $ then we have 
 $(x_1 x_k^{ \sigma_\mu -1} x_{k+1}^{p _{\mu+1 }-p_\mu -1} )^2 -x_0 \widetilde{{\rm Frob(S)}}\in I $ but $x_1^2-x_0^h x_2\in I$ so we get a contradiction.\\ 
 Finally in the case 6iii) we recall  that $\rho _{\mu }=\widetilde{\rho }=1$ so that $\rho _{\mu +1}=0$, which implies $s_{\mu +1}=0$ and 
 $x_{k+1}^{p _{\mu+1 }}-x_0\in I$, which is not possible.
\item  
If $\rho_\mu >1$ then we have $\rho_\mu> \rho _{\mu +1}$ and  $\widetilde{r}=r'_\mu-r'_{\mu +1}+h=h+1$, which is not possible since $r'_\mu\geq h, r'_{\mu +1}<0$.
\end{enumerate}
 
$\bullet$ If $M_1\in \widetilde{PF(S)_2} $ then the possible cases are 5i), 5ii), 5iii) or 6i), so $\rho_\mu =1$, 
$M_1=x_k^{ \sigma_\mu} x_{k+1}^{p _{\mu+1 }-p_\mu -1} $. Hence 
$$x_2 x_k^{ \sigma_\mu} x_{k+1}^{p _{\mu+1 }-p_\mu -1}-x_0x_1\in I, $$
but $x_2 x_k^{ \sigma_\mu}-x_0^{r'_\mu -h} x_1 x_{k+1}^{p_\mu}\in I$. It follows that
 $$x_0^{r'_\mu -h}  x_{k+1}^{p _{\mu+1 }-1}-x_0\in I, $$ which is not possible since $p _{\mu+1 }>1$.
 \end{enumerate}
\end{proof}

As a consequence we are reduced to study Almost symmetric AAG semigroups with $\widetilde{PF(S)_i}=\{\widetilde{{\rm Frob(S)}}\} $, for $i=1$ or $i=2$.
\end{subsection}
\begin{subsection}{Almost symmetric AAG-semigroups with $\widetilde{PF(S)_2}=\{\widetilde{\rm Frob(S)}\} $}
	
\begin{theorem}\label{almostsym1} Suppose $\widetilde{PF(S)_2}=\{\widetilde{\rm {Frob(S)}}\} $, $k\geq 3, t(S)\geq 2$ and either  $r'_\mu\geq h$ or $ \rho_\mu =0$. Then $S$ is almost symmetric  if and only if

\item (i)  Either 
$$a=(\sigma k+2)p-((\sigma-1)k+l), \ 
d=pr+\sigma,   c=((\sigma-1)k+l)r+(\sigma k+2)\sigma, 
$$ 
with $h=1,k\geq 3, 1\leq l\leq k-1$, $p,\sigma \geq 2, r<-(\sigma +1)$. Moreover $t(S)=k-l+1 $ and
$ \widetilde{PF(S)_1} =\{x_i  x_{k+1}^{p-1}\mid i=1,\dots ,k-l\}$, $ \widetilde{PF(S)_2} =\{x_1 x_k^{\sigma}x_{k+1}^{p-2}\}$, 
	
\item (ii) or $$a=(\sigma k+2)p-(\sigma k+1), \
d=pr+h(\sigma+1)+1,   c=(\sigma k+1)r+(\sigma k+2)(h(\sigma+1)+1), $$ 
with  $h\geq 1, \sigma\geq 1,p\geq 2, r>- h(\sigma+1)-1.$ Moreover $t(S)=2 $ and 
$ \widetilde{PF(S)_1} =\{x_{k+1}^{p-1}\}$, $ \widetilde{PF(S)_2} =\{x_1 x_k^{\sigma}x_{k+1}^{p-2}\}.$ 
\end{theorem}

\begin{proof} In order to prove this theorem, we often refer to the  cases of Theorem \ref {tPSF}. At first we  have by Theorem \ref {tPSF} that $ \widetilde{PF(S)_2} =\{L_\gamma  x_k^{{\sigma_\mu  }-\delta } x_{k+1}^{p_{\mu +1}-p_{\mu}-1}\}$ where $\gamma , \delta \in \{0,1\}$,  and  
$ \widetilde{PF(S)_1} =\{L_\Gamma  x_k^{\widetilde{\sigma }-\varepsilon } x_{k+1}^{p_{\mu +1}-1},\dots ,L_\Delta  x_k^{\widetilde{\sigma }-\varepsilon } x_{k+1}^{p_{\mu +1}-1}\}$ for some $\Gamma \leq \Delta $, $\varepsilon  \in \{0,1\}$. We set    $M_{\widetilde{\sigma }-\varepsilon}\defeq x_k^{\widetilde{\sigma }-\varepsilon } x_{k+1}^{p_{\mu +1}-1}$, $N_{\sigma_\mu -\delta }\defeq  x_k^{{\sigma_\mu  }-\delta } x_{k+1}^{p_{\mu +1}-p_{\mu}-1}$.   Since $S$ is almost symmetric we have 
\begin{equation}\label{caseI0} L_\Gamma L_\Delta M_{\widetilde{\sigma }-\varepsilon}^2-x_0 L_\gamma  N_{\sigma_\mu -\delta }\in I .\end{equation}
 
Note that $2p_{\mu +1}-2-(p_{\mu +1}-p_{\mu}-1)=p_{\mu +1}+p_{\mu}-1\geq 0  $  and 
$x_{k+1}^{p_{\mu +1}}-x_0 ^{-r'_{\mu +1}}L_{\rho _{\mu +1}}x_{k}^{\sigma _{\mu +1}}\in I$  so that (\ref{caseI0}) becomes 
\begin{equation}\label{caseIbis1}  x_0 ^{-r'_{\mu +1}}L_\Gamma L_\Delta L_{\rho _{\mu +1}}x_{k}^{\sigma _{\mu +1}+2\widetilde{\sigma }-2\varepsilon}x_{k+1}^{p_{\mu }-1} - x_0 L_\gamma  x_k^{{\sigma_\mu  }-\delta }\in I .\end{equation}

Since $L_\gamma  x_k^{\sigma_\mu -\delta} \in \widetilde{\Ap(S)}$ we have $r'_{\mu +1}\in\{0,-1\}$. We have to consider two case cases.\\
\begin{enumerate}
\item If $r'_{\mu +1}=0$ then $\rho _{\mu +1}>0, s_{\mu}-s_{\mu +1}>1$. The only possible cases for $\card( \widetilde{PF(S)_1} )$ are 1b),1d) or 1e).

\begin{enumerate}
\item In cases 1b),1d) we have $\widetilde{\rho }\in \{0,1\}, \Gamma=1, \Delta=k-\rho _{\mu +1}, \varepsilon =1$ so that  (\ref{caseIbis1}) becomes 
\begin{equation}\label{caseIbis2}  x_0 ^{h}x_1 x_{k}^{\sigma _{\mu +1}+2\widetilde{\sigma }-1}x_{k+1}^{p_{\mu }-1} - x_0 L_\gamma  x_k^{{\sigma_\mu  }-\delta }\in I .\end{equation}
which implies $h=1$. Let $\Theta =\sigma _{\mu +1}+2\widetilde{\sigma }-1-({\sigma_\mu  }-\delta )$. 
By definition $s_{\mu}-s_{\mu +1}=(\sigma _{\mu}-\sigma _{\mu +1})k+ \rho _{\mu}-\rho _{\mu +1}$. So if $\rho _{\mu}=\rho _{\mu +1}$ then $\widetilde{\sigma }= \sigma _{\mu}-\sigma _{\mu +1}, \widetilde{\rho }=\rho _{\mu}-\rho _{\mu +1}$. If   $\rho _{\mu}>\rho _{\mu +1}$ then $\widetilde{\sigma }= \sigma _{\mu}-\sigma _{\mu +1}$, and if $\rho _{\mu}<\rho _{\mu +1}$ then $\widetilde{\sigma }= \sigma _{\mu}-\sigma _{\mu +1}-1$, $\widetilde{\rho }=k+\rho _{\mu}-\rho _{\mu +1}$.\\ 
Suppose $\rho _{\mu}<\rho _{\mu +1}$. Then $\widetilde{\rho }\geq 1$ and we have equality if and only if $\rho _{\mu}=0, \rho _{\mu +1}=k-1$,  so we are in case 4i) with $k=2$, but we assume $k\geq 3$. \\
So in cases 1b),1d) we have  $\widetilde{\sigma }= \sigma _{\mu}-\sigma _{\mu +1}$, by a simple computation we have  $\Theta =\widetilde{\sigma }-1+\delta\geq 0$. So that  (\ref{caseIbis2}) becomes 
\begin{equation}\label{caseIbis3}  x_1 x_{k}^{\widetilde{\sigma }-1+\delta}x_{k+1}^{p_{\mu }-1} -  L_\gamma  \in I .\end{equation}
which is possible if and only if $\gamma =1,\widetilde{\sigma }-1+\delta=0,p_{\mu }=1$. In cases 1b),1d) we have $\widetilde{\sigma }>0$,  so  $\widetilde{\sigma }-1+\delta=0$ only if $\widetilde{\sigma }=1,\delta=0$ so we are in case 7i) with $\rho _{\mu}=2$. Since $r'_{\mu+1}=0$ we have $r_{\mu+1}=-\sigma_\mu$. Set $l=\rho _{\mu +1}$,  $\sigma\defeq \sigma_\mu,p\defeq p_{\mu+1},r\defeq r_{\mu}$ then we have the table 
$$\begin{tabular}{lll|l}
   s & p & r & r'\\
    \hline 
   $\cdots$& $\cdots$& $\cdots$&$\cdots$\\
   $\sigma k+2$ & 1 & $r$& \\
   $(\sigma  -1)k+l$ & $p$ & $-\sigma$ & $0$
\end{tabular}\\$$
where $l=2$ in case 1b), $l=1$ in case 1d). By Lemma 2.2.4 of \cite{Mo2}  we get\\
 $a=(\sigma k+2)p-((\sigma-1)k+l), \ 
d=pr+\sigma,  c=((\sigma-1)k+l)r+(\sigma k+2)\sigma, 
$\\
with $h=1,k\geq 3, 1\leq l\leq 2$, $p,\sigma \geq 2, r<-(\sigma +1)$. Moreover we have $t(S)=k-l+1 $.

\item In case 1e) we have $\widetilde{\rho }>1, \Gamma=1, \Delta=\min {\widetilde{\rho }-1, k-\rho _{\mu +1}}, \varepsilon =0$ so that  (\ref{caseIbis1}) becomes 
\begin{equation}\label{caseIbis4} x_1 x_\Delta x_{\rho _{\mu +1}}x_{k}^{\sigma _{\mu +1}+2\widetilde{\sigma }}x_{k+1}^{p_{\mu }-1} - 
x_0 L_\gamma  x_k^{{\sigma_\mu  }-\delta }  \in I .\end{equation}

Now we consider three cases.\\ 
If   $\rho _{\mu}=0$ then we are in case 4i), which implies $k=2$, contrary to our hypothesis.\\ 
 If   $\rho _{\mu}>\rho _{\mu +1}$ then $\widetilde{\rho}= \rho_{\mu}-\rho _{\mu +1}\leq k-\rho _{\mu +1}+1$, which implies $\Delta =\rho_{\mu}-\rho _{\mu +1}-1$.  So Equation (\ref{caseIbis4}) becomes
 \begin{equation}\label{caseIbis5} x_1 x_{\rho_{\mu}-\rho _{\mu +1}-1} x_{\rho _{\mu +1}}x_{k}^{\sigma _{\mu +1}+2\widetilde{\sigma }}x_{k+1}^{p_{\mu }-1} - 
x_0 L_\gamma  x_k^{{\sigma_\mu  }-\delta }  \in I ,\end{equation} but $x_1 x_{\rho_{\mu}-\rho _{\mu +1}-1} x_{\rho _{\mu +1}}-x_0^{2h}x_{\rho _{\mu}}\in I$ so we have 
 $x_0^{2h} x_{\rho_{\mu}} x_{k}^{\sigma _{\mu +1}+2\widetilde{\sigma }}x_{k+1}^{p_{\mu }-1} - x_0  L_\gamma x_k^{{\sigma_\mu  }-\delta }  \in I$, which is not possible since $h\geq 1$.\\

If $0<\rho _{\mu}<\rho _{\mu +1}$ then $\widetilde{\sigma }= \sigma _{\mu}-\sigma _{\mu +1}-1$, $\widetilde{\rho }=k+\rho _{\mu}-\rho _{\mu +1}\geq k-\rho _{\mu +1}+1$, which implies $\Delta =k-\rho _{\mu +1}$. Equation (\ref{caseIbis4}) becomes
\begin{equation}\label{caseIbis52} x_1 x_{k-\rho _{\mu +1}} x_{\rho _{\mu +1}}x_{k}^{\sigma _{\mu +1}+2\widetilde{\sigma }}x_{k+1}^{p_{\mu }-1} - 
x_0 L_\gamma  x_k^{{\sigma_\mu  }-\delta }  \in I .\end{equation}
But $x_{k-\rho _{\mu +1}} x_{\rho _{\mu +1}}-x_0^h x_k\in I$, hence we have 
\begin{equation}\label{caseIbis53}x_0^h x_1 x_{k}^{\sigma _{\mu +1}+2\widetilde{\sigma }+1}x_{k+1}^{p_{\mu }-1} - 
x_0 L_\gamma  x_k^{{\sigma_\mu  }-\delta }  \in I ,\end{equation}
which implies $h=1$. Note that since   $\sigma _{\mu +1}+2\widetilde{\sigma }+1-({\sigma_\mu  }-\delta )  =\widetilde{\sigma }+\delta\geq 0,$  we have 
\begin{equation}\label{caseIbis54} x_1 x_{k}^{\widetilde{\sigma }+\delta }x_{k+1}^{p_{\mu }-1} - 
 L_\gamma    \in I ,\end{equation} this is  possible case if and only if $\widetilde{\sigma }=\delta =0, \gamma=1,p_{\mu }=1$. So we are in case 7i) with $\rho_\mu=2$.  Since $r'_{\mu+1}=0$ we have $r_{\mu+1}=-\sigma_\mu$. We set $l=\rho _{\mu +1}$,  $\sigma\defeq \sigma_\mu,p\defeq p_{\mu+1},r\defeq r_{\mu}$. We have the table 
$$\begin{tabular}{lll|l}
   s & p & r & r'\\
    \hline 
   $\cdots$& $\cdots$& $\cdots$&$\cdots$\\
   $\sigma k+2$ & 1 & $r$& \\
   $(\sigma  -1)k+l$ & $p$ & $-\sigma$ & $0$
\end{tabular}$$
where $l>2$.  By Lemma 2.2.4 of \cite{Mo2}  we get\\
 $a=(\sigma k+2)p-((\sigma-1)k+l), \ 
d=pr+\sigma,  c=((\sigma-1)k+l)r+(\sigma k+2)\sigma, 
$\\
with $h=1,k\geq 3, 3\leq l\leq k-1$, $p,\sigma \geq 2, r<-(\sigma +1)$. Moreover $t(S)=k-l+1 $.
\end{enumerate}
\item If $r'_{\mu +1}=-1$ then 
 (\ref{caseIbis1}) becomes 
\begin{equation}\label{caseII0}  L_\Gamma L_\Delta L_{\rho _{\mu +1}}x_{k}^{\sigma _{\mu +1}+2\widetilde{\sigma }-2\varepsilon}x_{k+1}^{p_{\mu }-1} -  L_\gamma  x_k^{{\sigma_\mu  }-\delta }\in I .\end{equation}
We have to study the cases 2a), 2b), 2c), and 2d).
\begin{enumerate}
\item  In cases 2a), 2d)   we have $\Gamma , \Delta>0, \Gamma +\Delta \leq k$, so  $x_{\Gamma } x_{\Delta }-x_0^h x_{\Gamma+\Delta  }$ which implies 
$x_0^h x_{\Gamma+\Delta  } L_{\rho _{\mu +1}}x_{k}^{\sigma _{\mu +1}+2\widetilde{\sigma }-2\varepsilon}x_{k+1}^{p_{\mu }-1} -  L_\gamma  x_k^{{\sigma_\mu  }-\delta }\in I$ a contradiction. 
\item  In case 2b) we have $s_{\mu}-s_{\mu +1}=1, \Gamma=\Delta =0,\varepsilon =0$, so we have 
$L_{\rho _{\mu +1}}x_{k}^{\sigma _{\mu +1}}x_{k+1}^{p_{\mu }-1} -  L_\gamma  x_k^{{\sigma_\mu  }-\delta }\in I$. We discuss on $\rho _\mu $.\\
  If $\rho _\mu=0$ then $\rho _{\mu +1}=k-1$, we are in case 4i) $\card( \widetilde{PF(S)_2} )=1$  if and only if $k=2$ but we assume $k\geq 3$.\\ 
If $\rho _\mu\geq 1$ then $\rho _{\mu +1}=\rho _\mu-1, \sigma _{\mu}=\sigma _{\mu +1}$ so we have  
$L_{\rho _{\mu +1}}x_{k}^{\delta }x_{k+1}^{p_{\mu }-1} -  L_\gamma \in I$ which is possible if and only if $p_{\mu }=1$ and  case 7i) with $\rho _\mu=2.$

We set $\sigma\defeq \sigma_\mu,p\defeq p_{\mu+1},r\defeq r_{\mu}$. Note that  $-1=r'_{\mu+1 } =r_{\mu+1 }+ h(\sigma+1) $ so $r_{\mu+1 }=- h(\sigma+1)-1.$ We have  the table
$$\begin{tabular}{lll|l}
   s & p & r & r'\\
    \hline 
   $\cdots$& $\cdots$& $\cdots$&$\cdots$\\
   $\sigma k+2$ & 1 & $r$& \\
   $\sigma k+1$ & $p$ & $- h(\sigma+1)-1$ & $-1$
\end{tabular}$$
 By Lemma 2.2.4 of \cite{Mo2}  we get\\
 $ \phantom{(1)....    }\hskip 3cm a=(\sigma k+2)p-(\sigma k+1), \
d=pr+h(\sigma+1)+1, $\\ 
$ \phantom{(1)....    }\hskip 3cm c=(\sigma k+1)r+(\sigma k+2)(h(\sigma+1)+1), $\\
with  $h\geq 1, \sigma\geq 1,p\geq 2, r>- h(\sigma+1)-1,$  $t(S)=2$.

\item  In case 2c) we have $\widetilde{\rho }=1,\widetilde{\sigma }>0, \Gamma=1, \Delta =k,\varepsilon =1$, 
then 
 (\ref{caseII0}) becomes 
\begin{equation}\label{caseII1}  x_1 x_k L_{\rho _{\mu +1}}x_{k}^{\sigma _{\mu +1}+2\widetilde{\sigma }-2}x_{k+1}^{p_{\mu }-1} -  L_\gamma  x_k^{{\sigma_\mu  }-\delta }\in I .\end{equation}
Now we discuss on $\rho _{\mu +1}$.\\
If $\rho _{\mu +1}=0$ then $\rho _{\mu}=1, \widetilde{\sigma }=\sigma_\mu-\sigma _{\mu+1}$ which implies that we are in case 6ii). So that $\gamma =1,\delta =1$,   
$\sigma _{\mu +1}+2\widetilde{\sigma }-1-{\sigma_\mu  }+\delta= \widetilde{\sigma }$.  (\ref{caseII0}) becomes 
\begin{equation}\label{caseII2}  x_1 x_{k}^{\widetilde{\sigma }}x_{k+1}^{p_{\mu }-1} -  x_1 \in I .\end{equation} which is impossible since $ \widetilde{\sigma }>0$.\\
If $\rho _{\mu +1}>0$. Note that if   $\rho _{\mu}>\rho _{\mu +1}$ then  $\rho _{\mu +1}=\rho _{\mu}-1$, so we are in case 7i) with $\rho _{\mu}=2,\rho _{\mu +1}=1$, and we have  that
 (\ref{caseII0}) becomes   $x_1 x_k x_1 x_{k}^{\sigma _{\mu +1}+2\widetilde{\sigma }-2}x_{k+1}^{p_{\mu }-1} -  L_\gamma  x_k^{{\sigma_\mu  }-\delta }\in I$,  which leads to a contradiction since $x_1^2-x_0^h x_2\in I$.\\
 If   $\rho _{\mu}<\rho _{\mu +1}$ then $\widetilde{\rho }=1$ implies  $\rho _{\mu +1}=k-1,\rho _{\mu}=0$, so we are in case 4i), which implies $k=2$, contrary to our assumptions.
\end{enumerate}
\end{enumerate}
\end{proof}

\end{subsection}
\begin{subsection}{Almost symmetric AAG-semigroups with $\widetilde{\rm PF(S)_1}=\{\widetilde{\rm Frob(S)}\}$}
	
\begin{theorem}\label{almostsym2}Suppose $\widetilde{PF(S)_1}=\{\widetilde{\rm Frob(S)}\} $, $k\geq 3, t(S)\geq 2$ and either  $r'_\mu\geq h$ or $ \rho_\mu =0$. Then $S$ is almost symmetric if and only if  
\item (i) Either
$$a=(\sigma k+l+2)(p+1)-lp, \ 
d=-(p+1)\sigma-pr, \   
 c=-l\sigma -(\sigma k+l+2)r, 
$$ 
with $h=1,k\geq 4, \sigma \geq 1, 1\leq l\leq k-3, r<-\sigma$. Moreover  $t(S)=l+1$ and $ \widetilde{PF(S)_1} =\{x_1  x_k^{\sigma}  x_{k+1}^{p}\}$, $ \widetilde{PF(S)_2} =\{x_2  x_k^{\sigma } ,\dots ,x_{l+1}  x_k^{\sigma}  \}.$ 

\item (ii) Or 
$$a=(\sigma k)(p+1)-(k-2)p, \ 
d=(1- h\sigma)(p+1)-pr, c=(k-2)(1- \sigma) -\sigma k r, 
$$ 
with $h \geq  1,k\geq 3, p\geq 1,\sigma \geq 2, r<-h$. Moreover we have  $t(S)=k-1$ and $ \widetilde{PF(S)_1} =\{x_1  x_k^{\sigma -1 }  x_{k+1}^{p}\}$, $ \widetilde{PF(S)_2} =\{x_2  x_k^{\sigma -1} ,\dots ,x_{k -1}  x_k^{\sigma -1 }  \}$.

\item (iii) Or
$$a=(\sigma k+1)(p+1)-(k-1)p, \ 
d=(p+1)(1-h\sigma)-pr,   c=(k-1)(1-h\sigma) -(\sigma k+1)r, 
$$ 
with $h\geq 1,k\geq 3, p,\sigma \geq 1, r<\min {-h,1-h\sigma } $. Moreover  $t(S)=k$ and  $ \widetilde{PF(S)_1} =\{x_1  x_k^{\sigma -1 }  x_{k+1}^{p}\}$, $ \widetilde{PF(S)_2} =\{x_2  x_k^{\sigma -1} ,\dots , x_k^{\sigma }  \}$.
\item (iv) Or 
$$a=( k+1)(p+1)-kp =k+p+1, \ 
d=-(p+1)-pr, c=-k-(k+1)r, 
$$
with $h=1,k\geq 3,p\geq 1, r<-1$. Moreover  $t(S)=k+1$ and  $ \widetilde{PF(S)_1} =\{x_{k+1}^{p}\}$, $ \widetilde{PF(S)_2} =\{x_1 ,\dots ,x_k \}$.

\item (v) Or
$$a=( \sigma k+1)(p+1)-(2k-1)p, \ 
d=-(p+1)\sigma-pr,  c=-\sigma(2k-1)-(\sigma k+1)r, 
$$
with $h=1,k\geq 3,\sigma \geq 2,p\geq 1, r<-\sigma $. Moreover  $t(S)=2$ and $ \widetilde{PF(S)_1} =\{x_1  x_k^{\sigma - 2 }  x_{k+1}^{p}\}$, $ \widetilde{PF(S)_2} =\{x_1  x_k^{\sigma -1 }  \}$. 

\end{theorem}

\begin{proof} 
We have $ \widetilde{PF(S)_1} =\{L_\gamma  x_k^{\widetilde{\sigma }-\varepsilon } x_{k+1}^{p_{\mu +1}-1}\}$ where $\gamma , \varepsilon\in \{0,1\}$. By  Theorem \ref {tPSF}  we have $ \widetilde{PF(S)_2} =\{x_\Gamma  N_{\sigma_\mu -\delta },\dots ,x_\Delta  N_{\sigma_\mu -\delta } \}$ for some $\Gamma \leq \Delta , \delta \in \{0,1\}$, where $N_{\sigma_\mu -\delta }\defeq x_k^{\sigma_\mu -\delta  } x_{k+1}^{p_{\mu +1}-p_{\mu}-1}$. Set $M_{\widetilde{\sigma }-\varepsilon}\defeq x_k^{\widetilde{\sigma }-\varepsilon } x_{k+1}^{p_{\mu +1}-1}$.
Since $S$ is almost symmetric we have 
\begin{equation}\label{caseIter} L_\Gamma L_\Delta   N_{\sigma_\mu -\delta }^2-x_0 L_\gamma M_{\widetilde{\sigma }-\varepsilon}\in I.\end{equation}
Now we consider all the cases about  $\widetilde{PF(S)_2}$ in  Theorem \ref {tPSF}. 
\begin{enumerate}
\item Cases 4i) or 7i). We have  $\Gamma +\Delta \leq k.$ Note that $ \delta =1,\rho _\mu=0, \Gamma+\Delta=k, \sigma_\mu \geq 2 $ in case 4i) and $ \delta =0,\rho _\mu>1, \Gamma+\Delta=\rho _{\mu}  $ in case 7i).
Since $x_{\Gamma }x_{\Delta }-x_0^h x_{\Gamma+\Delta  }\in I$,   equation (\ref{caseIter}) becomes:
\begin{equation}\label{caseII2-6-a}
x_0 ^{h-1} x_{ \Gamma + \Delta }x_k^{2\sigma_\mu -2\delta  } x_{k+1}^{2(p_{\mu +1}-p_{\mu}-1)}- x_\gamma  x_k^{\widetilde{\sigma }-\varepsilon } x_{k+1}^{p_{\mu +1}-1} \in I.\end{equation}
Let $\Theta =2\sigma_\mu -1$ in case 4i) and $\Theta =2\sigma_\mu $ in case 7i). We can write 
\begin{equation}\label{caseII2-6-4}
x_0 ^{h-1} L_{\rho _\mu} x_k^{\Theta } x_{k+1}^{2(p_{\mu +1}-p_{\mu}-1)}- x_\gamma  x_k^{\widetilde{\sigma }-\varepsilon } x_{k+1}^{p_{\mu +1}-1} \in I.\end{equation}
Since $ \Theta \geq \sigma _\mu $ and  $L_{\rho _\mu} x_k^{\sigma_\mu} - x_0^{r'_{\mu} } x_{k+1}^{p_{\mu}}\in I$ we have 
$$x_0 ^{r'_{\mu} +h-1}
x_k^{\Theta - \sigma _\mu} x_{k+1}^{2(p_{\mu +1}-p_{\mu}-1)}- x_\gamma  x_k^{\widetilde{\sigma }-\varepsilon } x_{k+1}^{p_{\mu +1}-1}\in I,$$ with $r'_{\mu} +h-1\geq 1 $, but $x_\gamma  x_k^{\widetilde{\sigma }-\varepsilon } x_{k+1}^{p_{\mu +1}-1}\in \widetilde{\Ap(S,a_0)},$ a contradiction.

\item Case 4ii). We have  $\Gamma =\widetilde{\rho } ,\Delta=k-1, \delta =1,\rho _\mu=0, s _{\mu+1}< k-1.$ Note that  $s_{\mu +1}=\rho _{\mu +1}<k-1 $, so that $\widetilde{\rho }>1, r'_{\mu +1}<0,\sigma_\mu\geq 2,  \widetilde{\sigma }=\sigma_\mu -1\geq 1$. This implies that we are in case 2d) with $\widetilde{\rho }=2, \gamma =1, \varepsilon =0$. Hence $\Gamma +\Delta =k+1$.
Equation (\ref{caseIter}) becomes:
\begin{equation}\label{caseII2-6-a2} x_{1 }
x_k^{2\sigma_\mu -1 } x_{k+1}^{2(p_{\mu +1}-p_{\mu}-1)}-x_0 x_1  x_k^{\sigma_\mu -1} x_{k+1}^{p_{\mu +1}-1} \in I.\end{equation} 
But $x_k^{\sigma_\mu}-x_0^{r'_{\mu}} x_{k+1}^{p_{\mu}}\in I$ and $2(p_{\mu +1}-p_{\mu}-1)-p_{\mu}-(p_{\mu +1}-1)=p_{\mu +1}-p_{\mu}-1$ so we have 
$$ x_0^{r'_{\mu}-1} x_{k+1}^{p_{\mu +1}-p_{\mu}-1}-1\in I,$$ which is possible if and only if $r'_{\mu}=1, p_{\mu +1}=p_{\mu}+1$.  
Moreover we have $1=r'_{\mu } =r_{\mu }+ \sigma_\mu h$ so $r_{\mu}=1- \sigma_\mu h$. Set $\sigma \defeq \sigma_\mu, p\defeq p_\mu, r\defeq r_{\mu+1}$,  we get
$$ \widetilde{PF(S)_1} =\{x_1  x_k^{\sigma -1 }  x_{k+1}^{p}\} ,   \widetilde{PF(S)_2} =\{x_2  x_k^{\sigma -1} ,\dots ,x_{k -1}  x_k^{\sigma -1 }  \}$$
 and the  table
 $$\begin{tabular}{lll|l}
   s & p & r & r'\\
   \hline 
   $\cdots$& $\cdots$& $\cdots$&$\cdots$\\
   $\sigma k$ & $p$ & $1- \sigma h$& $1$\\
   $k-2$ & $p+1$ & $r$ &  $<0$\\
\end{tabular}$$
  By Lemma 2.2.4 of \cite{Mo2}  we get\\
$ \phantom{(1)....    }\hskip 0.2cm a=(\sigma k)(p+1)-(k-2)p, \ 
d=(1-\sigma h)(p+1)-pr, $\\
$ \phantom{(1)....    }\hskip 0.2cm c=(k-2)(1- \sigma h) -\sigma k r, 
$\\
with $h\geq 1,\sigma\geq 2,p\geq 1,r<-h,t(S)=k - 1$.
  
\item Cases 5i) or 6i). We have $\rho _{\mu}=1, \delta =1,$ and  $\Gamma +\Delta =k+1$.
Equation (\ref{caseIter}) becomes:
\begin{equation}\label{caseII2-6-b} x_{1 }
x_k^{2\sigma_\mu -1 } x_{k+1}^{2(p_{\mu +1}-p_{\mu}-1)}-x_0 L_\gamma  x_k^{\widetilde{\sigma }-\varepsilon} x_{k+1}^{p_{\mu +1}-1} \in I.\end{equation} 
But $\sigma_\mu\geq 1$, $x_{1 }x_k^{\sigma_\mu}-x_0^{r'_{\mu}} x_{k+1}^{p_{\mu}}\in I$ and $2(p_{\mu +1}-p_{\mu}-1)-p_{\mu}-(p_{\mu +1}-1)=p_{\mu +1}-p_{\mu}-1$, so we have 
$$ x_0^{r'_{\mu}-1} x_k^{\sigma_\mu-1}x_{k+1}^{p_{\mu +1}-p_{\mu}-1}-L_\gamma  x_k^{\widetilde{\sigma }-\varepsilon}\in I,$$ which implies $r'_{\mu}=1=h$. Now we are in case 6i), so we have $s_{\mu}-s_{\mu +1}=1$. Thus $\rho _{\mu +1}=0$ and the only possible  case is 2b) with $\sigma_\mu=1$. Moreover we have $1=r'_{\mu } =r_{\mu }+ \sigma_\mu+1= r_{\mu }+2$ so $r_{\mu}=-1 $. Set $p\defeq p_\mu, r\defeq r_{\mu+1}$, we have 
$ \widetilde{PF(S)_1} =\{x_{k+1}^{p}\}$, $ \widetilde{PF(S)_2} =\{x_1 ,\dots ,x_k \}$.
  $$\begin{tabular}{lll|l}
   s & p & r & r'\\
    \hline 
   $\cdots$& $\cdots$& $\cdots$&$\cdots$\\
   $  k+1$ & $p$ & $-1 $& $1$\\
   $k$ & $p+1$ & $r$ &$<0$ \\
\end{tabular}$$
  By Lemma 2.2.4 of \cite{Mo2}  we get\\
$ \phantom{(1)....    }\hskip 0.2cm a=( k+1)(p+1)-kp =k+p+1, \ 
d=-(p+1)-pr, $\\
$ \phantom{(1)....    }\hskip 0.2cm c=-k-(k+1)r,  
$\\
with $h=1,k\geq 3,p\geq 1, r<-1$, $t(S)=k+1$.

\item Case 5ii). We have $\rho _{\mu}=1,\delta =1,1<s_{\mu+1}=\rho _{\mu+1}<k, r'_\mu > h$ and  $\Gamma +\Delta =k+\widetilde{\rho }\geq k+1$.
If $\widetilde{\rho }=1$ then $\rho _{\mu+1}=0$ a contradiction. So we have $\widetilde{\rho }\geq 2$. Since $\rho _{\mu+1}>1$ we are in case 2d), which implies  $\widetilde{\rho }= 2, s_{\mu+1}=\rho _{\mu+1}=k-1,\varepsilon =0,\gamma =1.$ Note also that $\widetilde{\sigma }=\sigma_\mu -1$.
Equation (\ref{caseIter}) becomes:
\begin{equation}\label{caseII2-6-b1} x_{2 }
x_k^{2\sigma_\mu -1 } x_{k+1}^{2(p_{\mu +1}-p_{\mu}-1)}-x_0 x_1  x_k^{\sigma_\mu -1} x_{k+1}^{p_{\mu +1}-1} \in I.\end{equation} 
But $\sigma_\mu\geq 1$, $x_{2 }x_k^{\sigma_\mu}-x_0^{r'_{\mu}-h}x_1  x_{k+1}^{p_{\mu}}\in I$ and $2(p_{\mu +1}-p_{\mu}-1)-p_{\mu}-(p_{\mu +1}-1)=p_{\mu +1}-p_{\mu}-1,$ so we have $x_0^{r'_{\mu}-h-1}x_{k+1}^{p_{\mu +1}-p_{\mu}-1}-1\in I$, which is possible if and only if $r'_{\mu}=h+1, p_{\mu +1}-p_{\mu}-1=0$. 
Since  $h+1=r'_{\mu } =r_{\mu }+ h(\sigma+1) ,$   we have $r_{\mu}=1- h\sigma$, on the other hand   $r' _{\mu+1 }=r _{\mu+1 }+h<0 $, so $r_{\mu+1 }<-h$. Set $\sigma \defeq \sigma_\mu, p\defeq p_\mu, r\defeq r_{\mu+1}$,  we have
$$ \widetilde{PF(S)_1} =\{x_1  x_k^{\sigma -1 }  x_{k+1}^{p}\}, \widetilde{PF(S)_2} =\{x_2  x_k^{\sigma -1} ,\dots , x_k^{\sigma }  \},t(S)=k.$$ 
 $$ \begin{tabular}{lll|l}
   s & p & r & r'\\
    \hline 
   $\cdots$& $\cdots$& $\cdots$&$\cdots$\\
   $\sigma k+1$ & $p$ & $1- h\sigma $&$h+1$\\
   $k-1$ & $p+1$ & $r$ & $<0$\\
\end{tabular}$$
  By Lemma 2.2.4 of \cite{Mo2}  we get\\
$ \phantom{(1)....    }\hskip 0.2cm a=(\sigma k+1)(p+1)-(k-1)p, \ 
d=(p+1)(1-h\sigma)-pr, $\\
$ \phantom{(1)....    }\hskip 0.2cm c=(k-1)(1-h\sigma) -(\sigma k+1) r, $\\
with  $h\geq 1,p\geq 1, \sigma \geq 1, r< -h$.
 
\item Case 5iii) We have  $\Gamma +\Delta=0$,  $\rho _{\mu}=1,s _{\mu+1}=1$ so $\widetilde{\rho}=0$ and $r'_{\mu+1}<0$. The possible case of Theorem \ref {tPSF}  with $\card(\widetilde{PF(S)_1})=1  $  satisfying $\widetilde{\rho}=0$ is  2a) with  $k=2$, it is impossible since we suppose $k\geq 3$.
 
\item Case 6ii) We have  $\Gamma =\Delta=1, \delta =1,\rho _\mu=1,h=r'_\mu, s _{\mu+1}>k, s_{\mu}-s_{\mu+1}>1$, so Equation (\ref{caseIter}) becomes:
\begin{equation}\label{caseII-14-1} x_0 ^{h-1}x_{2}
x_k^{2\sigma_\mu -2} x_{k+1}^{2(p_{\mu +1}-p_{\mu}-1)}-L_\gamma  x_k^{\widetilde{\sigma }-\varepsilon } x_{k+1}^{p_{\mu +1}-1} \in I.\end{equation}
which implies $h=1.$ If $\gamma =0$ then we are in case 2b) and   $\widetilde{\sigma }=\varepsilon=0$ which implies  $x_{k+1}^{p_{\mu +1}-1} \in \In(I)$, which is not possible. 
Thus we have  $\gamma =1,h=1$ and 
\begin{equation}\label{caseII-1} x_{2}
x_k^{2\sigma_\mu -2} x_{k+1}^{2(p_{\mu +1}-p_{\mu}-1)}-x_1  x_k^{\widetilde{\sigma }-\varepsilon } x_{k+1}^{p_{\mu +1}-1} \in I.\end{equation}
If $\sigma _{\mu}=1$ then $s_{\mu}=k+1$ which implies $s_{\mu+1}<k$ a contradiction. So we have  $\sigma _{\mu}\geq 2$. 
From the \gbb of $I$, we have
$x_2 x_k^{\sigma_\mu } -x_1 x_{k+1}^{p_{\mu}} \in I$, note also that $2p_{\mu +1}-p_{\mu}-2-(p_{\mu +1}-1)=p_{\mu +1}-p_{\mu}-1\geq 0$. So equation (\ref{caseII-1}) becomes:
$$
x_k^{\sigma_\mu -2 } x_{k+1}^{p_{\mu +1}-p_{\mu}-1}-x_k^{\widetilde{\sigma }-\varepsilon }  \in I,$$  which is possible if and only if 
$p_{\mu +1}=p_{\mu}+1, \sigma_\mu -2-(\widetilde{\sigma }-\varepsilon)=0$. We have to examine two cases:\\
If  $\varepsilon =1$ then  the possible cases are 1b), 1d) or 2a). The case 2a) implies $k=2$ and the case 1b) implies  $\widetilde{\rho }=0, \rho  _{\mu+1}=1$ so $k=2$. On the other hand the case 1d) implies $\widetilde{\rho }=1, \rho  _{\mu+1}=k-1,$ but $\rho  _{\mu}=1$ which gives $\widetilde{\rho }=2$, a contradiction.\\
If $ \varepsilon=0$, then the possible cases are 1e),  or 2d). In both cases we have $\widetilde{\rho }=2, \rho  _{\mu+1}=k-1$, so   
$\sigma_\mu =\sigma _{\mu+1}+\widetilde{\sigma }+1$. We have $\sigma_\mu -2-(\widetilde{\sigma }-\varepsilon)= \sigma_\mu -2-(\sigma_\mu-\sigma _{\mu+1}-1)=0$ if and only if $\sigma _{\mu+1}=1$. In the case  1e) we have $r' _{\mu+1}=0$, which implies $r_{\mu+1}=-2. $ Since $h= 1$ we can suppose $d>0$, which implies $r_{\mu}>r_{\mu+1 }$, that is $-\sigma _{\mu}>-2$, but we have $\sigma _{\mu}\geq 2$, a contradiction, we can exclude the case 1e). The only possible case is 2d). Set $\sigma \defeq \sigma_\mu, p\defeq p_\mu, r\defeq r_{\mu+1}$, we have the table
$$\begin{tabular}{lll|l}
   s & p & r & r'\\
    \hline 
   $\cdots$& $\cdots$& $\cdots$&$\cdots$\\
   $ \sigma k+1$ & $p$ & $-\sigma  $& $1$\\
   $2k-1$ & $p+1$ & $r$ &$<0$ \\
\end{tabular}$$
 Apply Lemma 2.2.4 of \cite{Mo2},  we get\\
$ \phantom{(1)....    }\hskip 1cm a=( \sigma k+1)(p+1)-(2k-1)p, \ 
d=-(p+1)\sigma-pr, $\\
$ \phantom{(1)....    }\hskip 1cm c=-\sigma(2k-1)-(\sigma k+1)r, 
$\\
with $h=1,k\geq 3,\sigma \geq 2,p\geq 1, r<-\sigma $ and $t(S)=2$.

\item Case 7ii). We have  $\Gamma +\Delta=\widetilde{\rho }+\rho _\mu-1$ and  $\delta =0,s _{\mu+1}=\rho _{\mu+1}< \rho _\mu-1$, which implies  $\widetilde{\rho }\geq 2$, $r'_{\mu+1 }<0$ and  $\sigma_\mu =\widetilde{\sigma }$. The only possible case is 2d) with $\widetilde{\rho }=2$, so we have 
$\widetilde{\rho }+\rho _{\mu}-1=\rho _{\mu}+1\leq k$. From the \gbb of $I$ we have $x_2 x_{\rho _\mu-1 }-x_0^h x_{\rho _\mu+1 },
x_{\rho _\mu+1 }
x_k^{\sigma_\mu  } -x_{0}^{r'_{\mu}-h} x_{1}x_{k+1}^{p_{\mu}} \in I$. So Equation (\ref{caseIter}) becomes: 
\begin{equation}\label{caseII-1-b2}x_{0}^{r'_{\mu}-1} x_{1}x_{k+1}^{p_{\mu}}
x_k^{\sigma_\mu  } x_{k+1}^{2(p_{\mu +1}-p_{\mu}-1)}-x_1  x_k^{\sigma_\mu-\varepsilon } x_{k+1}^{p_{\mu +1}-1} \in I.\end{equation}
which implies $r'_{\mu}=h=1$. Since $2p_{\mu +1}-p_{\mu}-2-(p_{\mu +1}-1)=p_{\mu +1}-p_{\mu}-1\geq 0$ we get  
\begin{equation}\label{caseII-1-b3} 
x_k^{\varepsilon } x_{k+1}^{p_{\mu +1}-p_{\mu}-1}-1  \in I.\end{equation}
\\
This is possible if and only if $ \varepsilon =0, p_{\mu +1}=p_{\mu}+1$.
Since $ h=1$ we can suppose $d>0$, so the sequence $r_i$ is strictly decreasing. We have $1=r'_{\mu } =r_{\mu }+ \sigma_{\mu}+1 $ so $r_{\mu}=- \sigma_{\mu}, r_{\mu+1}<- \sigma_\mu  $. Set $\sigma =\sigma _\mu , p=p_{\mu}, l=s_{\mu +1}, r\defeq r_{\mu+1} $, then we have the table:
 \begin{center} \begin{tabular}{lll|l}
   s & p & r & r'\\
    \hline 
   $\cdots$& $\cdots$& $\cdots$&$\cdots$\\
   $\sigma k+l+2$ & $p$ & $- \sigma $& $1$\\
   $l$ & $p+1$ & $r $ & $<0$\\
\end{tabular}
\end{center} 
with $\sigma \geq 1,1\leq l<k-2, r<- \sigma $. We have $t(S)=l+1$.  By Lemma 2.2.4 of \cite{Mo2}  we get\\
$ \phantom{(1)....    }\hskip 0.2cm a=(\sigma k+l+2)(p+1)-lp, \ 
d=-(p+1)\sigma-p r,$\\
$ \phantom{(1)....    }\hskip 0.2cm c=-l\sigma -(\sigma k+l+2)r.  $
 \end{enumerate}
\end{proof}
\end{subsection}

\section{Formula for Frobenius number of  Almost Symmetric AAG-semigroups}
This section extends and generalizes all the results of \cite{R-G}.

 \begin{theorem}\label{almostsym1-2} Let $S$ be an AAG almost symmetric with $k\geq 3, t(S)\geq 2$ and either  $r'_\mu\geq h$ or $ \rho_\mu =0$. Then there is  a linear or quadratic formula for the Frobenius number in terms of $a,d,c,k$ and the type $t(S)$.
\end{theorem}
\begin{proof} We will use the notations and the proofs of section 5. At first we study the cases where $\widetilde{\rm Frob(S)}\in \widetilde{PF(S)_i} , \card(\widetilde{PF(S)_i})\geq 2$, for some $i=1,2$.

In  Theorem \ref{almostsym00} we have that $F(S)= kd$, so there is nothing to prove in this case. 

In  Theorem \ref{almostsym0} we have 
$F(S)= a_1+(\sigma -1 )a_k+ (p-2)c-a$. From Equation (\ref{neg-case1}) we get $a_2+(p -1)c=a+a_1 $. But 
 $x_{k+1}^{p } - x_0 x_k^{\sigma}\in I$, so we have $p  c=a+ \sigma a_k$. By a simple computation we get
 $$F(S)=3a_1-2a_2-a_k.$$
 So we can assume that  $\widetilde{PF(S)_i} =\{\widetilde{\rm Frob(S)}\}$ for some $i=1,2$. 

\item{I) }  Suppose $\widetilde{PF(S)_2} =\{\widetilde{\rm Frob(S)}\}$. We have by Theorem \ref{almostsym1} that    $S$ is almost symmetric with $k\geq 3, t(S)\geq 2$ if and only if we have either one of the following conditions.
\begin{enumerate} 
\item $ \widetilde{PF(S)_1} =\{x_{k+1}^{p-1}\}$, $ \widetilde{PF(S)_2} =\{x_1 x_k^\sigma x_{k+1}^{p-2}\}, t(S)=2$. 
 Since $S$ is almost symmetric we have  $x_{k+1 }^{2(p-1)}-x_0 \widetilde{\rm Frob(S)}\in I,$
 so 
  $F(S)=2(p-1)c- 2a$, hence determining a formula for $F(S)$ consists of determining $p$ in terms of $a,d,c,k$.\\ 
We have $x_{k+1 }^p - x_0x_{1 }x_k ^{\sigma}\in I$. 
So
\begin{equation}\label{for405} c p=a+a_{1}+\sigma a_k.\end{equation}
Since $a=(\sigma k+2)p-(\sigma k+1),$ we have  \begin{equation}\label{for406}   \sigma k(p -1)=a+1-2p.  \end{equation}
 We multiply (\ref{for405}) by $k(p -1)$ and by using (\ref{for406}) we get 
 \begin{equation}\label{for4066} k(p -1)p c=k(p -1)(a+a_{1 })+  a_k( a+1-2p).  \end{equation}
 So we get a second order equation in the variable $p$ 
 \begin{equation}\label{for4067} kcp^2 -(k(c+a+a_{1})-2a_k  )p  -   a_k( a +1)+   k (a+a_{1 })=0,  \end{equation}
 so $$p= \displaystyle \frac{k(c+a+a_1) -2a_k+ \sqrt{(k(c+a+a_{1})-2a_k  )^2- 4kc(-   a_k( a +1)+   k (a+a_{1 }))}}{2kc}.$$
 
\item $ \widetilde{PF(S)_1} =\{x_i  x_{k+1}^{p-1}\mid i=1,\dots ,k-l\}$,$ \widetilde{PF(S)_2} =\{x_1 x_k^{\sigma}x_{k+1}^{p-2}\}$, $t(S)=k-l+1$ and $ h=1$. 
Since $S$ is almost symmetric we have  $x_1 x_{k-l}x_{k+1 }^{2(p-1)}-x_0 \widetilde{\rm Frob(S)}\in I$, 
  so  $  x_{k-l+1}x_{k+1 }^{2(p-1)}- \widetilde{\rm Frob(S)}\in I$
and   $F(S)=2(p-1)c+a_{k-l+1} - a$,  hence determining a formula for $F(S)$ consists of determining $p$ in terms of $a,d,c,k, l$. 
We have $x_{k+1 }^p - x_{l }x_k ^{\sigma-1}\in I$, hence 
\begin{equation}\label{for37} cp=a_{l }+\sigma a_k -a_k.\end{equation}
Since $a=(\sigma k+2)p-((\sigma-1)k+l)$ we have  \begin{equation}\label{for38}  \sigma k(p -1)=a-2p-k+l.  \end{equation}
 We multiply (\ref{for37}) by $k(p -1)$ and by using (\ref{for38}) we get 
 \begin{equation}\label{for39} k(p -1)p c=k(p -1)a_{l }+  a_k( a-2p-k+l) -k(p -1)a_k. \end{equation}
 So we get a second order equation in the variable $p$ 
 \begin{equation}\label{for40} kcp^2 -(k(c+a_{l }-a_k )  -2a_k )p  -   a_k( a +l)+   k a_{l }=0,  \end{equation}
 so $p= \displaystyle \frac{k(c+a_{l }-a_k )  -2a_k+ \sqrt{(k(c+a_{l }-a_k )  -2a_k )^2- 4kc(-   a_k( a +l)+   k a_{l })}}{2kc}.$
\end{enumerate}

\item{II) } Suppose $\widetilde{PF(S)_1} =\{\widetilde{\rm Frob(S)}\}$. We have by Theorem \ref{almostsym2} that   $S$ is almost symmetric with $k\geq 3, t(S)\geq 2$ if and only if either 
\begin{enumerate}
 \item $ \widetilde{PF(S)_1} =\{x_1  x_k^{\sigma}  x_{k+1}^{p-1 }\}$, $ \widetilde{PF(S)_2} =\{x_2  x_k^{\sigma } ,\dots ,x_{l+1}  x_k^{\sigma}  \}$, $t(S)=l+1 , h=1,l\leq k-3$. 
Since $S$ is almost symmetric we have  $x_2  x_{l+1}  x_k^{2\sigma}-x_0 \widetilde{Frob(S)}\in I$, so 
 we have $   x_{l+3}  x_k^{2\sigma}-  \widetilde{\rm Frob(S)}\in I$
and   $F(S)=a_{l+3}+2\sigma a_k- a$. On the other hand  $x_{l+2}  x_k^{\sigma}-x_0x_{k+1 }^{p-1} \in I $, 
 so 
\begin{equation}\label{for45}a_{l+2}+\sigma a_k=a+ c(p-1).\end{equation} Consequently   
  $F(S)=a_{l+3}+2(a+ c(p-1)- a_{l+2} )- a$. Determining a formula for $F(S)$ consists of determining $p $ in terms of $a,d,c,k,l$.
 Since $a=(\sigma k+l+2) p-l(p-1)$ we have \begin{equation}\label{for46}   \sigma kp=a-2p-l.  \end{equation}
 We multiply (\ref{for45}) by $kp$ and by using (\ref{for46}) we get 
 \begin{equation}\label{for47} ka_{l+2}p+k\sigma a_k p=kap+ kcp(p-1). \end{equation}
 So we get a second order equation in the variable $p$ 
 \begin{equation}\label{for401} kcp^2 -(k(c-a+a_{l+2})-2a_k  )p  -   a_k( a-l)=0, \end{equation}
 so $p= \displaystyle \frac{k(c-a+a_{l+2})-2a_k  + \sqrt{(k(c-a+a_{l+2})-2a_k  )^2+ 4kc a_k( a-l)}}{2kc}.$
 
\item $ \widetilde{PF(S)_1} =\{x_1  x_k^{\sigma -1 }  x_{k+1}^{p}\}$, $ \widetilde{PF(S)_2} =\{x_2  x_k^{\sigma -1} ,\dots ,x_{k -1}  x_k^{\sigma -1 }  \}$, $t(S)=k-1$ and $h = 1$.  
Since $S$ is almost symmetric we have $x_2 x_{k -1} x_k^{2\sigma-2}-x_0 \widetilde{\rm Frob(S)}\in I$
 that is $   x_{1}  x_k^{2\sigma-1}- x_0 \widetilde{\rm Frob(S)}\in I.$
So  $F(S)=a_{1}+2\sigma a_k-a_k- 2a$. On the other hand  $   x_k^{\sigma}-x_0x_{k+1 }^{p-1} \in I $, 
 so we have 
\begin{equation}\label{for451} \sigma a_k=a+ c(p-1).\end{equation} Consequently   
  $F(S)=a_{1}+2(a+ c(p-1))-a_k- 2a$. Determining a formula for $F(S)$ consists of determining $p $ in terms of $a,d,c,k$.
 Since $a=(\sigma k)(p+1)-(k-2)p$ we have \begin{equation}\label{for461}   \sigma k p=a +kp-2p-k-2. \end{equation}
 We multiply (\ref{for451}) by $kp$ and by using (\ref{for461}) we get 
 \begin{equation}\label{for471}  a_k( a +kp-2p-k-2)=kap+ kcp(p-1). \end{equation}
So we get a second order equation in the variable $p$ 
 \begin{equation}\label{for402} kcp^2 -(k(c-a+a_k)-2a_k  )p  -   a_k( a-k+2)=0,  \end{equation}
 and $p= \displaystyle \frac{k(c-a+a_k)-2a_k   + \sqrt{(k(c-a+a_k)-2a_k )^2+ 4kc  a_k( a-k+2)}}{2kc}.$
 
\item  $ \widetilde{PF(S)_1} =\{x_1  x_k^{\sigma -1 }  x_{k+1}^{p}\}$, $ \widetilde{PF(S)_2} =\{x_2  x_k^{\sigma -1} ,\dots , x_k^{\sigma }  \}$, $t(S)=k$.  
Since $S$ is almost symmetric we have  $x_2 x_k^{2\sigma-1}-x_0 \widetilde{\rm Frob(S)}\in I$, 
 so we have  $F(S)=a_{2}+2\sigma a_k-a_k- 2a$. On the other hand $   x_1x_k^{\sigma}-x_0^{h+1}  x_{k+1 }^{p-1} \in I $, 
 so we have 
\begin{equation}\label{for571} a_1+\sigma a_k=(h+1)a+ c(p-1)\end{equation} and  
  $F(S)=a_{2}+2(a+ c(p-1)-a_1)-a_k- 2a$. Determine a formula for $F(S)$ consist to determine $p $ in terms of $a,d,c,k$.
 By developing the formula $a=(\sigma k+1)(p+1)-(k-1)p$ we have \begin{equation}\label{for581}   \sigma k p=a+p(k-2)-k+1.  \end{equation}
 We multiply (\ref{for571}) by $kp$ and by using (\ref{for581}) we get 
 \begin{equation}\label{for481} ka_1p+ a_k(a+p(k-2)-k+1)=k(h+1)ap+ kcp(p-1). \end{equation}
So we get a second order equation in the variable $p$ 
 \begin{equation}\label{for491} kcp^2 -(k(a_1+c-(h+1)a+a_k)-2a_k)p  -   a_k( a-k+1)=0,  \end{equation} and
 $$p= \displaystyle \frac{k(a_1+c-(h+1)a+a_k)-2a_k + \sqrt{(k(a_1+c-(h+1)a+a_k)-2a_k)^2+ 4kc  a_k( a-k+1)}}{2kc}.$$   
             
\item $ \widetilde{PF(S)_1} =\{x_{k+1}^{p}\}$, $ \widetilde{PF(S)_2} =\{x_1 ,\dots ,x_k \}$, $t(S)=k+1$.
We have 
$F(S)= p c-a$ and $a =k+p+1$ so $F(S)= c( a-k-1)-a$. 

\item $ \widetilde{PF(S)_1} =\{x_1  x_k^{\sigma - 2 }  x_{k+1}^{p}\}$, $ \widetilde{PF(S)_2} =\{x_1  x_k^{\sigma -1 }  \}$, $t(S)=2$. 
Since $x_1^2 x_k^{2\sigma-2}-x_0 \widetilde{Frob(S)}\in I$ and  $x_1^2-x_0 x_2\in I$ 
 we have $F(S)=a_{2}+2\sigma a_k-2a_k- a$. From $ x_1x_k^{\sigma}-x_0 x_{k+1 }^p \in I $,  
 we get 
\begin{equation}\label{for58} a_1+\sigma a_k=a+ c(p-1).\end{equation} Hence  
  $F(S)=a_{2}+2(a+ c(p-1)-a_1 )-2a_k- a$. Determine a formula for $F(S)$ consist to determine $p $ in terms of $a,d,c,k$.
  
  By developing the formula  $a=( \sigma k+1)(p+1)-(2k-1)p$ we have \begin{equation}\label{for59}   \sigma k p=a+p(2k-2)-2k+1 . \end{equation}
 We multiply (\ref{for58}) by $kp$ and by using (\ref{for59}) we get 
 \begin{equation}\label{for60} ka_1p+ a_k(a+p(2k-2)-2k+1)=kap+ kcp(p-1). \end{equation}
So we get a second order equation in the variable $p$ 
 \begin{equation}\label{for61} kcp^2 -(k(c+d+2a_k)-2a_k)p  -   a_k( a-2k+1)=0 , \end{equation} and 
 $p= \displaystyle \frac{k(c+d+2a_k)-2a_k+ \sqrt{(k(c+d+2a_k)-2a_k)^2+ 4kc  a_k( a-2k+1)}}{2kc}.$
 \end{enumerate}
\end{proof}

\begin{remark}{\rm Given a AAG-semigroup $S$ with data $a,d,c,h,k$  we can determine if $S$ is almost symmetric satisfying the condition $r_\mu\geq h$ or $\rho _\mu=0$ by using the Euclidean algorithm for the GCD of three numbers as implemented in \cite{Mo3}, which gives us the table described in Lemma \ref{equations-s-p}. The table gives us the sequences $s_i,p_i,r_i, q_i$ and the number $\mu $, so that we can use our theorems in section 5 to check if the semigroup is almost symmetric or not. A faster algorithm consists of using Theorem \ref{almostsym1-2}, which 
 gives us  a system of $2k+2$ formulas for $p$, if $p>1$ is a natural number solution of one of them, then we also have the probably type of $S$, so the number $l$, which allows us to determine the possible case  of section 5. By using the formulas for $a$ in the corresponding case of section 5 we can find $\sigma $, and if   $\sigma $ is a natural number, then eventually solve the natural number $r$ from the formula for $d$ and finally check if we have the right value of $c$.}
\end{remark}
\begin{example}We have implemented the above algorithm. For 
$$150\leq   a\leq  165, -5\leq   d\leq 10,170  \leq   c\leq 186, 19 \leq k\leq  20,1\leq h\leq 4,$$ the  AAG-semigroup  such that $r_\mu\geq h$  is almost symmetric  
if and only if either
\item $1)\ a=  155, d= 1 ,  c= 177,  k=  20,  h=  4$. \\
 Theorem \ref{almostsym1} case 1),  $ p=  8, \sigma= 1, r = -1,  t(S)= 2$, or
\item $2)\ a=  163, d=  -2 ,  c= 170,  k=  19,  h=  1$. \\
 Theorem \ref{almostsym1} case 2),  $l = 14, p=  3, \sigma= 4, r = -2,  t(S)= 6$, or
\item $3)\ a=  165,d=  4, c= 170,k=  19, h=  1$. \\
  Theorem \ref{almostsym2} case 1),  $l=5, p=  4, \sigma= 2,r = -4, t(S)=  6$, or
\item $4)\ a=  165, d=  -2 ,  c= 174,  k=  19,  h=  1$. \\
Theorem \ref{almostsym1} case 2),  $l = 12, p=  3, \sigma= 4, r = -2,  t(S)= 8$, or
\item $5)\ a=  163, d=  7,  c= 179, k=  19, h=  1$. \\
Theorem \ref{almostsym2} case 5),  $p=  6, \sigma= 3, r= -5, t(S)=  2$, or
\item $6)\ a=  165, d=  7,  c= 183, k=  19, h=  3$.\\ 
Theorem \ref{almostsym2} case 3),   $p=  7, \sigma= 2, r = -7, t(S)=  19$, or 
\item $7)\ a=  165, d=  -1, c= 186, k=  19, h=  4$. \\
Theorem \ref{almostsym2} case 3),  $p=  7, \sigma= 2, r = -8, t(S)=  19$. 
\end{example}


\begin{thebibliography}{GHS} 
 \bibitem{B-F}{\sc V. Barucci, R. Fröberg}, {\it One-Dimensional Almost Gorenstein Rings},  Journal of Algebra {\bf 188} (1997), 418--442. 
 
\bibitem{Br} {\sc  H. Bresinsky},  {\it Symmetric semigroups of integers generated by 4 elements}, Manuscripta Math.
{\bf  17} (1975), 205-219.


\bibitem{E} {\sc Kazufumi Eto}, \textit{ Almost Gorenstein monomial curves in affine four space}. Journal of Algebra, {\bf 488}  (2017), 362- 387.




\bibitem{GRR} {\sc Ignacio García-Marco, Jorge L. Ramírez Alfonsín, Øystein J. Rødseth}, \textit{ Numerical semigroups II:  Pseudo-symmetric AA-semigroups},
Journal of Algebra, {\bf 470} (2017), 484-498.



\bibitem{Mo1} {\sc Marcel Morales},  {\it Syzygies of monomial curves and a linear Diophantine problem of Frobenius}, Preprint. Max Planck Institut fur Mathematik (Bonn-RFA) (1987).

\bibitem{Mo2}  {\sc Marcel Morales}, {\it Equations des variétés  monomiales en codimension deux}, Journal of Algebra {\bf 175} (1995), 1082-1095. 

\bibitem{Mo3}  {\sc Marcel Morales}, {\it Software Frobenius number-\gb  basis} download at https: //www-fourier.ujf-grenoble.fr/~morales/.

\bibitem{M-D} {\sc Marcel Morales, Nguyen Thi Dung},  {\it A "pseudo-polynomial" algorithm for  the  Frobenius number and Gr\"obner basis}, Journal of Symbolic Computation, https://doi.org/10.1016/j.jsc.2023.102233.

\bibitem{Mos}   {\sc A. Moscariello}, {\it On the type of an almost Gorenstein monomial curve,}  J. Algebra 456 (2016), 266-277. 


\bibitem{N}  {\sc H. Nari},  {\it  Symmetries on almost symmetric numerical semigroups}, Semigroup Forum, 86.
(2013), 140 - 154
 

\bibitem{P} {\sc Dilip P. Patil}, \textit{Minimal sets of generators for the relation ideals of certain monomial curves}, Manuscripta Math. {\bf 80} (1993), 239- 248 .



\bibitem{RR}   {\sc J. L. Ram\'irez Alfons\'in,  and  O. J. R\o dseth}, {\it Numerical Semigroups, Ap\'ery set and Hilbert series}, Semigroup Forum, {\bf  79(2)} (2009), 323-340. 

 

\bibitem{ROD2}{\sc J. R\o dseth}, {\it  On a linear Diophantine problem of Frobenius II}, J. Reine Angew. Math. {\bf 307/308} (1979), 431-440.
\bibitem{R-G}{\sc J.C. Rosales, and P.A. García-Sánchez}, {\it Pseudo-symmetric numerical semigroups with three generators}  (2005). Published by Elsevier Inc.DOI: 10.1016/j.jalgebra.2005.06.005.


\end{thebibliography}
\end{document}